\documentclass[12pt]{amsart}

\usepackage{amsmath, amssymb, amsthm, latexsym, amscd}
\usepackage[all, knot]{xy}
\xyoption{arc}
\usepackage{graphicx}

 \newlength{\baseunit}               
 \newcount{\numlines}                
\numberwithin{equation}{section}
\newtheorem{theorem}[equation]{Theorem}
\newtheorem{prop}[equation]{Proposition}
\newtheorem{lm}[equation]{Lemma}
\newtheorem{cor}[equation]{Corollary}

\theoremstyle{definition}
\newtheorem{definition}[equation]{Definition}

\newtheorem{remark}[equation]{Remark}

\newtheorem{examplesremarks}[equation]{Examples and Remarks}

\DeclareMathOperator{\Hom}{Hom}
\DeclareMathOperator{\Mor}{Mor}

\DeclareMathOperator{\im}{im}
\DeclareMathOperator{\coker}{coker}
\DeclareMathOperator{\Ind}{Ind}
\DeclareMathOperator{\Pro}{Pro}

\DeclareMathOperator{\pre}{pre}

\DeclareMathOperator{\IP}{IP}
\DeclareMathOperator{\IPs}{IP^s}
\DeclareMathOperator{\PrI}{PI}

\DeclareMathOperator{\id}{id}
\DeclareMathOperator{\Fun}{Fun}
\DeclareMathOperator{\BiFun}{Fun}

\newcommand{\fset}{{\bf Set}_0}
\newcommand{\set}{{\bf Set}}
\newcommand{\Zee} {\mathbb{Z}}

\newcommand{\C}{\mathcal C}
\newcommand{\Func}{\Fun(\Zee_+,\C)}
\newcommand{\Funct}{\Fun'(\Zee_+,\C)}
\newcommand{\Bifunc}{\BiFun(\Pi, \C)}
\newcommand{\Bifuncl}{\BiFun(\Lambda, \C)}
\newcommand{\Funcop}{\Fun(\Zee^{op}_+,\C)}
\newcommand{\Funcopt}{\Fun'(\Zee^{op}_+,\C)}
\newcommand{\Funcneg}{\Fun(\Zee_-,\C)}
\newcommand{\Funcnegt}{\Fun'(\Zee_-,\C)}
\newcommand{\rar}{\rightarrow}
\newcommand{\lar}{\leftarrow}

\newcommand{\lrar}{\longrightarrow}

\newcommand{\hrar}{\hookrightarrow}

\newcommand{\epi}{\twoheadrightarrow}

\newcommand{\limi}{\underset{i\in I}{``\varinjlim"}}
\newcommand{\limj}{\underset{j}{\varinjlim}}
\newcommand{\limproi}{\underset{i\in I}{``\varprojlim"}}
\newcommand{\limXi}{\underset{i\in I}{``\varinjlim"} X_i}
\newcommand{\limXZ}{\underset{i\in \Zee_+}{``\varinjlim"} X_i}
\newcommand{\limXphi}{\underset{i\in \Zee_+}{``\varinjlim"} X_{\phi(i)}}
\newcommand{\limXpsi}{\underset{i\in \Zee_+}{``\varinjlim"} X_{\psi(i)}}
\newcommand{\limXii}{\underset{i'\in I'}{``\varinjlim"} X_{i'}}
\newcommand{\limXj}{\underset{j\in J}{``\varinjlim"} X_j}
\newcommand{\limYi}{\underset{i\in I}{``\varinjlim"} Y_i}
\newcommand{\limYj}{``\varinjlim_{j\in J}" Y_j}
\newcommand{\limYij}{\underset{j}{``\varinjlim"}\underset{i}{``\varprojlim"}Y_{i,j}}
\newcommand{\limXij}{\underset{j}{``\varinjlim"}\underset{i}{``\varprojlim"}X_{i,j}}
\newcommand{\limsigma}{\underset{j}{``\varinjlim"}\underset{i\in\Sigma_j}{``\varprojlim"}X_{i,j}}
\newcommand{\limpi}{\underset{j}{``\varinjlim"}\underset{i\in\Pi_j}{``\varprojlim"}X_{i,j}}
\newcommand{\proYj}{\underset{j\in J}{``\varprojlim"} Y_j}
\newcommand{\proVvj}{\underset{j\in J}{``\varprojlim"} V'_j}
\newcommand{\proXi}{\underset{i\in I}{``\varprojlim"}X_i}
\newcommand{\proXii}{\underset{i'\in I'}{``\varprojlim"}X_{i'}}

\newcommand{\proindrigaa}{$\underset{j\in J}{\varinjlim}\underset{i\in I}{\varprojlim}\ X_{i,j}$}
\newcommand{\indhom}{\underset{i\in I}{\varprojlim}\underset{j\in J}{\varinjlim}\ Hom_{\C}(X_{i}, Y_{j})}
\newcommand{\prohom}{\underset{j\in J}{\varprojlim}\underset{i\in I}{\varinjlim}Hom_{\C}(X_{i}, Y_{j})}
\newcommand{\limcom}{\underset{k}{\varprojlim}\underset{j}{\varinjlim}Hom_{\C}(Y_{j}, Z_{k})}
\newcommand{\limcoma}{\underset{k}{\varprojlim}\underset{j}{\varinjlim}\underset{\alpha}{\varprojlim}Hom_{\C}(Y_{\alpha j}, Z_{k})}
\newcommand{\limcomak}{\underset{k}{\varprojlim}\underset{\alpha}{\varprojlim}\underset{j}{\varinjlim}Hom_{C}(Y_{\alpha j}, Z_{k})}
\newcommand{\limcomjka}{\underset{\alpha}{\varprojlim}\underset{k}{\varprojlim}\underset{j}{\varinjlim}Hom_{C}(Y_{\alpha j}, Z_{k})}
\newcommand{\indprorigaa}{\underset{j}{``\varinjlim"}``\underset{i}{\varprojlim"}X_{i,j}}

\newcommand{\AP}{\BiFun(\Pi, \A)}
\newcommand{\APa}{\BiFun^a(\Pi, \A)}
\newcommand{\EP}{\E(\Pi, \A)}
\newcommand{\limA}{\underleftrightarrow{\lim \ }\mathcal A}

\newcommand{\prelim}{\pre\underleftrightarrow{\lim \ }}
\newcommand{\xrar}{\xrightarrow}
\newcommand{\xlar}{\xleftarrow}
\newcommand{\A}{\mathcal A}
\newcommand{\B}{\mathcal B}

\newcommand{\E}{\mathcal E}
\newcommand{\F}{\mathcal F}
\newcommand{\Pp}{\mathcal P}

\newcommand{\vect}{{\bf Vect}_0(k)}
\newcommand{\vects}{{\bf Vect}(k)}

\newcommand{\limiti}{\underset{i}{\varinjlim}}

\newcommand{\liminvj}{\underset{j}{``\varprojlim"}}
\newcommand{\limitinvj}{\underset{j}{\varprojlim}}

\newcommand{\limij}{\underset{i}{``\varinjlim"}\underset{j}{``\varprojlim"}}
\newcommand{\limitij}{\underset{i}{\varinjlim}\ \underset{j}{\varprojlim}}

\newcommand{\proVj}{\underset{j\in J}{``\varprojlim"}V_j}

\newcommand{\dlim}{\underleftrightarrow{\lim}}

\newcommand{\T}{\mathcal T}
\newcommand{\Ll}{\mathcal L}

\begin{document}
\pagestyle{plain} 
\title{Locally Compact Objects in Exact Categories}
\author{Luigi Previdi}
\date{}
\maketitle
\begin{abstract}
We identify two categories of locally compact objects on an exact category $\A$. They correspond to the well-known constructions of the Beilinson category $\limA$ and the Kato category $\kappa(\A)$. We study their mutual relations and compare the two constructions. We prove that $\limA$ is an exact category, which gives to this category a very convenient feature when dealing with K-theoretical invariants, and study the exact structure of the category $\dlim\vect$ of Tate spaces. It is natural therefore to consider the Beilinson category $\limA$ as the  most convenient candidate to the role of the category of locally compact objects over an exact category. We also show that the categories $\Ind_{\aleph_0}(\C)$, $\Pro_{\aleph_0}(\C)$ of countably indexed ind/pro-objects over any category $\C$ can be described as  localizations of categories of diagrams over $\C$.
\end{abstract} 

\tableofcontents

\section{Introduction}
When dealing with the categorical formulation of some infinite-dimensional problems arising from different contexts in Analysis, Topology, Algebraic Geometry and Algebra, it is often natural to use the formalism of ind-pro objects of a certain category. This formalism was introduced by Grothendieck and his school in the early '60 (see \cite{sga}) and provided a general framework to address many questions arising in Algebraic Geometry. In the '80s, K. Kato took a further step, and considered iterated categories of ind/pro objects, and was able to express topological concepts in a more general and convenient context than that of a topological space (\cite{k}). In particular, it is a theorem of Kato that locally linearly compact topological vector spaces are just particular ind-pro objects over the category of finite dimensional vector spaces. In a similar vein, Kapranov has proved that totally disconnected locally compact topological spaces can be constructed also as a category of ind-pro objects of the category of finite sets  (see Theorem \eqref{loc}).

\vspace{0.3cm}

The question of finding an appropriate category-theoretic concept for the general concept of a locally compact space, arising from different research areas was addressed in the same period of time by A. Beilinson, in his paper \cite{B}. Precisely, Beilinson deals with the problem of characterizing ``local compactness" over an exact category, while Kato considers general categories. Beilinson's approach can thus be interpreted as the linear point of view about local compactness, while Kato's  construction in the years has proved particularly fruitful when dealing with analytical problems, for it was his construction that allowed M. Kapranov to address the basic problem of  creating a convenient framework for Harmonic Analysis over a 2-dimensional local field, overcoming analytical difficulties that appeared insurmountable, in \cite{ka}.

\vspace{0.3cm}

More recently,  three papers have appeared (\cite{op}, \cite{o} and \cite{ak})  that deal, in different contexts, with the category of locally compact objects of an exact category. In particular, in the first paper the authors use the language and the techinique of iterated ind/pro objects to describe familiar spaces of analytical functions and distributions as particular iterated ind/pro objects over an exact category (such as that of finite-dimensional vector spaces). In the second paper, the author introduces some class of infinite-dimensional vector spaces, the $C_n$-spaces, whose construction is very close to the iterated categories $\limA$; and in the last paper the authors use explicitly  the category of locally compact objects of an exact category to define the concept of $n$-Tate spaces.  It thus seems important to give a systematic treatment of local compactness in a category, and to compare the different definitions of local compactness thus far proposed when the
 ambient category is exact.

\vspace{0.3cm}

The aim of this work is to clarify the mutual relation between the two construction of locally compact objects, ``\`a la Kato" and ``\`a la Beilinson", at least when the base category $\A$ is an exact category. We give, in Theorem  \eqref{charac} a precise statement characterizing the Beilinson category $\limA$ in terms of the Kato category $\kappa(\A)$. We also prove the exactness of all the categories $\Ind(\A), \  \Pro(\A)$ and all the iterated categories $\Ind\Pro(\A), \  \Pro\Ind(\A)$, ..., and finally prove the exactness of $\limA$ using the technique of iterated ind/pro-objects, by showing the existence of a closed embedding of $\limA$ into $\Ind\Pro(\A)$. As a byproduct of this study, we have also found an alternative definition of the categories $\Ind_{\aleph_0}(\C)$, $\Pro_{\aleph_0}(\C)$, of countably indexed ind/pro-objects of any category $\C$. It turns out that they can be described as certain localizations of appropriately defined categories of diagrams $X_0\rar X_1\rar X_2\rar\cdots$ (resp., $X_0\lar X_1\lar X_2\lar\cdots$). The localization follows the pattern suggested by the approach of Beilinson, but it is applicable to all categories. The exactness of $\limA$ using the language of ind/pro objects makes this category the ``ideal candidate" to the role of the ``category of locally compact objects of an exact category". This allows one to study $\limA$ from the point of view of   algebraic K-theory. Partial results in this direction are obtained in the  forthcoming paper \cite{p}.

\vspace{0.3cm} 

An important example of an exact category is the category $\A=\Pp(R)$, 
the category of finitely generated projective modules over a ring $R$. In this case, an alternative approach to local compactness has been developed by V. 
Drinfeld in his paper \cite{dr}, where the notion of {\it Tate R-module} is defined. When $R$ is commutative, the concept of Tate $R$-module is understood as ``family of Tate spaces", and it appears that, for $R=k$ a 
field, Drinfeld's construction reduces to our $\dlim\vect$.

\vspace{0.3cm}

{\bf Acknowledgements.} This paper is part of the dissertation presented by the author to the Faculty of the Graduate School of Yale University in partial fulfillment of the requirements for the Degree of Doctor of Philosophy in Mathematics. The author wishes to thank his advisor, Professor Mikhail Kapranov,  for his assistance, and Professor Alexander Beilinson, for his remarks on several topics here discussed. The author also would like to express his gratitude to Professors Howard Garland, Igor Frenkel and Gregg Zuckerman, for the many discussions had with them concerning the results here expounded.

\section{Preliminary facts about Ind/Pro categories}

\subsection{Generalities}

Let $\C$ be a category. An {\it inductive system} over $\C$ \cite{sga} is a covariant functor
\[
X\colon I\to\C
\]
from a small filtering category $I$ to $\C$. We shall also use the notation 
$\{X_i\}_{i\in I}$. We shall refer to the objects $X_i$ as the {\it components} of $X$. If $u: i\rar i'$ is a morphism of $I$, we shall call the induced morphism $X(u): X(i)\lrar X(i')$ a {\it structure morphism} of the ind-system $X$. 

\vspace{0.1cm}

The {\it ind-object} associated to an ind-system $\{X_i\}$ is the formal symbol
$$
X=\limXi.
$$
Ind-objects of $\C$ are made into a category, by putting:
\begin{equation}\label{indhom}
\Hom_{\Ind(C)}(\limXi, \limYj):=\indhom .
\end{equation}
Let us consider the datum consisting, for each $i\in I$, of a choice of a $j=j(i)\in J$ and a  morphism of $\C$, $f^i_j: X_i\lrar Y_j$ compatible with the structure maps of the two ind-objects $X$ and $Y$. A morphism of 
ind-objects $f\colon X\to Y$ is thus an equivalence class of such data, 
under the equivalence relation induced by forming the limit as in 
\eqref{indhom}.
The maps $f^i_j$ shall be called the {\it components} of $f$. Notice that such 
components are not uniquely determined by $f$. In particular, when $I=J$ and $j(i)=i$, the morphism $f$ is a natural transformation and it is called a {\it straight} morphism.

\par

The composition of two morphisms of $\Ind(\C)$ is defined by composing the components in the obvious way. In this way the collection of all ind-objects of $\C$ with their morphisms becomes a category $\Ind(\C)$.

\par

\vspace{0.3cm}

The category $\Pro(\C)$ is formally defined as
$$
\Pro(\C):=\Ind(\C^{op})^{op}.
$$
Its objects are formal symbols 
$$
Y=\proYj 
$$
for contravariant functors
$$
Y:J^{op}\lrar\C
$$
from a small filtering category $J$ to $\C$. We shall call these objects {\it pro-objects} or {\it projective systems} over $\C$.

\vspace{0.2cm}

Let be $X=\{X_i\}_{i\in I}$ and $Y=\{Y_j\}_{j\in J}$ two pro-objects. Dualizing \eqref{indhom} we see that the class of {\it morphisms of pro-objects} is then given as

\begin{equation}\label{prohom}
\Hom_{\Pro(\C)}(\proXi, \proYj)=\prohom.
\end{equation}
Thus, given the datum consisting, for all $j\in J$, of an object $i=i(j)\in I$ and a morphism of $\C$, $f^i_j: X_i\lrar Y_j$ compatible with the structure maps of the two pro-objects $X$ and $Y$, a morphism $f$ of pro-objects $X=\proXi$ and $Y=\proYj$ is an equivalence class of such data, under the 
equivalence relation induced by forming the limit as in \eqref{prohom}. \\
Straight morphisms are defined in the same way as for $\Ind(C)$. The composition is still defined componentwise. For further details, we refer to \cite{sga}, Expos\'e 1, or \cite{am}, Appendix.

\par

We now introduce some further terminology.

\begin{definition}\label{indstrict} An ind-object $X$ is called {\it strict} if it can be represented by $X=\limXi$ where the structure morphisms are monomorphisms. Similarly, a pro-object $Y$ is called {\it strict} if it can be represented by $Y=\proYj$, where the structure morphisms are epimorphisms. 

\par

We denote by $\Ind^s(\C)$ (resp., $\Pro^s(\C)$) the full subcategory of $\Ind(\C)$ (resp., $\Pro(\C)$) whose objects are strict ind-objects (resp., strict pro-objects). 
\end{definition}

\begin{definition} $\Ind_{\aleph_0}(\C)$, $\Pro_{\aleph_0}(\C)$ are the categories of {\it countable ind-objects} (resp., pro-objects) of $\C$, i.e. those ind-objects (pro-objects) obtained as ind-limits (pro-limits) from a countable filtering category.
\end{definition}

\vspace{0.3cm}

Let be $\set$ the category of sets;  $\vects$ the category of vector spaces over a field $k$, $\fset$  the category of finite sets and $\vect$ the category of finite-dimensional vector spaces over a field $k$. Then, the proof of the following is elementary:

\begin{prop} There are the following equivalences of categories:
$$
\Ind(\fset)\cong\set;
$$
$$
\Ind(\vect)\cong\vects
$$
\end{prop}
Given two filtering cateogries $I$ and $J$, a {\it cofinal functor} $\phi: I\rar J$ is a functor satisfying  the following conditions: (1) for all objects $j\in J$, there is an object $i\in I$ for which $\Hom(j, \phi(i))\neq\emptyset$ and (2) for all object $i\in I$ and for each pair of morphisms $f,g: j\rar \phi(i)$ in $J$, there exists a morphism $h:i\rar i'$ in $I$ such that $\phi(h)f=\phi(h)g$.

\vspace{0.2cm} 

A {\it cofinal subcategory} of a filtering category $I$ is a full subcategory 
$I'$ of $I$, such that the embedding is a cofinal functor.
\vspace{0.1cm}
The following will be often useful:

\begin{lm}\label{cofinal}
Let I' be a cofinal subcategory of I. Then, in $\Ind(\C)$ we have $\limXi=\limXii$ and in $\Pro(\C)$ we have 
$\proXi=\proXii$.
\end{lm}

\begin{lm}\label{straight} (Straightification of morphisms) Let be $f: X\lrar Y$ a morphism in $\Ind(\C)$. Then it is possible to express $X$ and $Y$ as ind-systems $X=\limXi$ and $Y=\limYi$ with the same  category of indexes $I$, and $f$ as a natural transformation of functors. Similarly for morphisms of $\Pro(\C)$.
\end{lm}
For the proof of the above lemma, cf. \cite{am}.

\vspace{0.2cm}

From now on, we will consider only filtering categories which are partially ordered sets, called {\it preorders} in the sequel. This is not really restrictive 
since every filtering category $I$ admits a cofinal functor from a filtering preorder (cf. \cite{sga}, I.8.1). If $I$ and $J$ are such sets, then a functor $\phi\colon I\to J$ is a monotonic nondecreasing map. In particular, when $I=J=\Zee_+$, then $\phi:\Zee_+\rar\Zee_+$ is cofinal if and only if $\phi$ is monotonic, nondecreasing and $\lim_{n \to\infty}\phi = \infty$.

\subsection{Ind/Pro categories and localization of categories}

Let $\C$ be any category. In this section, motivated by the work of Beilinson (\cite{B}), we prove that the category  $\Ind_{\aleph_0}(\C)$ is the localization of the category of $\Zee_+$-indexed inductive systems on $\C$ modulo an 
equivalence relation, and similarly for $\Pro_{\aleph_0}(\C)$. We refer to \cite{gm}, Sect. III,2 for preliminary material about the localization of categories.

\vspace{0.3cm}

Let be $\Zee_+$ the preorder of nonnegative integers (considered as a category in the usual way). Let be $\phi: \Zee_+\rar\Zee_+$ a cofinal map. If $\psi$ is another such map, we write $\phi\leq \psi$ whenever, for all $i\in\Zee_+$, we have $\phi(i)\leq\psi(i)$.

\par

If $X$ is any object in the category of functors $\Func$, and $\phi\leq\psi$, we have a natural transformation of functors $X\cdot\phi\rar X\cdot\psi$ defined, for all $i$, by $X_{\phi(i)}\rar X_{\psi(i)}$.

\vspace{0.2cm}

For any pair of objects $X, Y$ of $\Func$, let us consider the equivalence 
relation $\sim$, which we will also denote $\sim_{X, Y}$, in $\Hom(X,Y)$ defined as follows: 
$f\sim_{X,Y} g$ if there exists an $i_0\in\Zee_+$ such that for all $i\geq i_0$ 
we have $f_i=g_i$. When $X$ and $Y$ are clear from the context, we write 
$f\sim g$. It is evident that this relation is compatible with the composition 
of morphisms.
We denote the quotient category $\dfrac{\Func}{\sim}$ by 
$\Funct$.

\vspace{0.2cm}

Next, consider the following class of morphisms in $\Funct$:
\begin{multline}
S:=\biggl\{Y\cdot\phi'\lrar Y\cdot\phi \text{ such that } Y\in\Funct; \ \phi, \phi': \Zee_+\rar\Zee_+ \text{ are cofinal, } \\  
 \text{ and }\phi'\leq\phi \biggr\}
\end{multline}
\begin{prop}\label{Slocal}
$S$ is a localizing system of morphisms in $\Funct$.
\end{prop}

\begin{proof}
In order to prove that $S$ is a localizing system, we have to check the following: \\
(a) $S$ contains $1_X$ for all $X\in\Funct$, and $S$ is closed by composition of morphisms, whenever the composition is defined. \\
(b) For all $f\in\Mor(\Funct)$, and $s\in S$, there is $g\in\Mor(\Funct)$ and $t\in S$, such that the square
$$
\xymatrix{ 
Y\ar[d]_{s}
\ar[r]^{f}
& X \ar[d]^{t} \\
\tilde{Y}\ar[r]_{g}
& \tilde{X}
}
$$
is commutative. \\ 
(c) For any pair of morphisms $f, g: X\lrar Y$ in $\Mor(\Funct)$, the existence of $s\in S$ with $s\cdot f=s\cdot g$ is equivalent to the existence of $t\in S$ with $f\cdot t=g\cdot t$.

\vspace{0.2cm} 

Condition (a) is clear. To prove (b), let $f:Y\rar X$  be given. Let be $s$ the natural transformation $Y\lrar \tilde Y$. Without loss of generality we can suppose that there exists a cofinal $\phi$ with $\id\leq\phi$ and that $s$ is the natural transformation $Y\rar Y\cdot\phi$. Then, for $g:=f\cdot\phi$, the diagram
$$
\xymatrix{ 
Y\ar[d]_{s}
\ar[r]^{f}
& X \ar[d]^{t} \\
Y\cdot\phi\ar[r]_{g}
& X\cdot\phi
}
$$
is commutative since it just expresses that $f$ is a natural transformation of functors, and (b) is proved. 

\par

For (c), let $\phi$ be cofinal. We write in components the equation $s\cdot f=s\cdot g$. We get for all $ i\in\Zee_+$ the commutative diagram

\begin{equation}\label{extensionphi}
\xymatrix{ 
X_i\ar[d]_{g_i}
\ar[r]^{f_i}
& Y_i \ar[d]^{s_i} \\
Y_i\ar[r]_{s_i}
& Y_{\phi(i)}.
}
\end{equation}
We want to define a cofinal functor $\psi:\Zee_+\rar\Zee_+$ such that, for all $j\in\Zee_+$, we have
\begin{equation}\label{extensionpsi}
\xymatrix{ 
X_{\psi(j)}\ar[d]_{t_j}
\ar[r]^{t_j}
& X_j \ar[d]^{g_j} \\
X_j\ar[r]_{f_j}
& Y_j
}
\end{equation}
The map $\psi$ is defined as follows.  

\vspace{0.1cm}

(1) If $j\in\im(\phi)$, define:
$$
\psi(j):=\max\{i_0: \phi(i_0)=j\}.
$$

(2) If $j\notin\im(\phi)$, and if there exists a largest integer 
$j_0\leq j$ which is in the image of $\phi$, define
$$
\psi(j):=\max\{i_0: \phi(i_0)=j_0\}.
$$

\vspace{0.1cm}

(3) If $j\notin\im(\phi)$, and it does not exist any integer $\leq j$ which is in 
the image of $\phi$, define 
$$
\psi(j):=0.
$$
It is clear that $\phi$ cofinal implies that $\psi$ is well defined and it is cofinal $\Zee_+\rar\Zee_+$. Moreover, if $\id_{\Zee_+}\leq\phi$, then $\psi\leq\id_{\Zee_+}$. We now prove that with this choice of $\psi$, for every $j$ we have a commutative diagram of type \eqref{extensionpsi}.

\vspace{0.2cm}

Suppose first that $j\in\im(\phi)$. Put $\psi(j)=i_0$, so that in particular 
$\phi(i_0)=j$. We have the commutative diagram

\begin{equation}\label{extension1}
\xymatrix{ 
X_{i_0}\ar[d]_{g_{i_0}}
\ar[r]^{f_{i_0}}
& Y_{i_0} \ar[d]^{s} \\
Y_{i_0}\ar[r]_{s}
& Y_{\phi(i_0)}.
}
\end{equation}
Since $\id\leq\phi$, and from naturality of both $f$ and $g$, we get the commutative diagrams
\begin{equation}\label{natural}
\xymatrix{
X_{i_0} \ar[d]_t\ar[r]^{f_{i_0}} & Y_{i_0}\ar[d]^s  && X_{i_0}  \ar[d]_t\ar[r]^{g_{i_0}} & Y_0\ar[d]^s \\
X_j\ar[r]_{f_j} & Y_j && X_j\ar[r]_{g_j} & Y_j\\
}
\end{equation}
so we obtain that \eqref{extension1} is equivalent to \eqref{extensionpsi} via the two diagrams \eqref{natural}.

\vspace{0.2cm}

Suppose next that $j\notin\im(\phi)$, and there is a largest $j_0\leq j$ which 
is in $\im(\phi)$. Then, $\psi(j)=\psi(j_0)=i_0$. So it is $\phi(i_0)=j_0$. By hypotheses, we have thus the commutative diagram \eqref{extension1}. Being $j_0\leq j$, we can compose the map $s$ with the structure map of $Y$: $Y_{j_0}\rar Y_j$. We use again the naturality of $f$ and $g$ to conclude that \eqref{extension1} is equivalent to \eqref{extensionpsi}. 

\vspace{0.2cm}

Finally, suppose $j\notin\im(\phi)$, and there is no integer $\leq j$ in the image of $\phi$. The same argument shows now that there is an integer 
$i_0$ such that for all $j\geq i_0$ we have $f_jt_j=g_jt_j$. It follows 
$f\cdot t=g\cdot t$ in $\Funct$. The proposition is proved.
\end{proof}
In particular, let us denote by $\Funct[S^{-1}]$ the localization of the category $\Funct$ by the localizing system $S$. We can describe this category by using the category of $S$-roofs over $\Funct$ (\cite{gm},  Lemma III,8).  

\vspace{0.2cm}

Thus, an object of $\Funct[S^{-1}]$ is an object of $\Func$, and a morphism $X\rar Y$ of $\Funct[S^{-1}]$ is an equivalence class of ``roofs", i.e. a pair $([f], \phi)$, where $\phi$ is cofinal $\Zee_+\rar\Zee_+$, and $[f]$ is an 
equivalence class of natural transformations $X\rar Y\cdot\phi$, where two roofs $([f],\phi), ([g],\psi)$ are equivalent if and only if there exists a third roof $([h], \theta)$ forming a commutative diagram of the form
\begin{equation}\label{eqroof}
\xymatrix{ 
&&(Y\psi)\theta&& \\
& Y\phi\ar[ru]^h && Y\psi\ar[lu]& \\
X  \ar[ru]^f \ar[rrru]_g   &&&& Y.\ar[lllu]\ar[lu]
}
\end{equation}
In particular, two morphisms $([f], \phi)$ and $([g], \psi)$ are equivalent as $S-$roofs if and only if there exists a cofinal $\theta:\Zee_+\rar\Zee_+$ such that $[f]$ and $[g]$ induce the same morphism on the objects $X_i\rar Y_{\theta(i)}$. Notice that it is $\theta(i)\geq\max(\phi(i), \psi(i))$.

\begin{theorem}\label{indlocal} 
$\Ind_{\aleph_0}(\C)$ is equivalent to $\Funct[S^{-1}]$.
\end{theorem}

\begin{proof}
Let's prove the existence of a functor $\Funct[S^{-1}]\rar \Ind_{\aleph_0}(\C)$. From the universal property of the localization functor 
$\Funct\rar\Funct[S^{-1}]$, it is sufficient to prove the existence of a functor $H: \Funct\rar\Ind_{\aleph_0}(\C)$ sending each $s\in S$ into an isomorphism.

\par

Define $H$ on the objects as $H(X):=\limXZ$, and extend it to the morphisms in the obvious way. $H$ so defined is clearly a functor. Let be $s: X\phi\rar X\psi$ a morphism in $S$. Then, we get:
$$
H(s): H(X\phi)=\limXphi\stackrel{\sim}\lrar\limXpsi=H(X\psi)
$$
Since both $\phi$ and $\psi$ are cofinal, the morphism $s$ in components is then a collection of structure maps $X_{\phi(i)}\rar X_{\psi(i)}$ of the ind-system of $\limXi$, hence it is the identity of this object in $\Ind_{\aleph_0}(\C)$. $H$ thus determines a functor $\Funct[S^{-1}]\rar \Ind_{\aleph_0}(\C)$, that we still denote by $H$. We show that $H$ is full, faithful and essentially surjective, and thus an equivalence. 

\vspace{0.2cm}

(i) $H$ is full (i.e. surjective on the Hom-sets). Let be $f:X\rar Y$ a morphism in $\Ind_{\aleph_0}(\C)$. The straightification of $f$ explained in Lemma \eqref{straight} allows to write $f$ as a natural transformation of the functors $X=\{X_i\}_{i\in\Zee_+}$ and $Y=\{Y_i\}_{i\in\Zee_+}$: then we have $H(f)=f$, i.e. $H$ is full.

\vspace{0.2cm}

(ii) $H$ is faithful (i.e. injective on the Hom-sets): Let be 
$f, g\in\Mor(\Funct[S^{-1}])$, with $H(f)=H(g)$. From the above roof description of morphisms,  we can describe these two morphisms as pairs  $(f,\phi), (g,\psi)$, where $\phi$ and $\psi$  are as in the definition of $S$, such that for all $i\in\Zee_+$, $f_i:X_i\rar Y_{\phi(i)}$, $g_i:X_i\rar Y_{\psi(i)}$ are morphisms of $\C$ which commute with the structure maps of $X$ and $Y$.

\par

On the other hand, $H(f)=H(g)$, i.e.  $\limi f_i=\limi g_i$ are equal in $\indhom$, as in equation \eqref{indhom}. Thus, for all $i$, the equivalence classes of $[f_i]$ and $[g_i]$ coincide in 
$\limj\Hom_{\C}(X_i, Y_j)$. This implies that there exists a $j=j(i)$ such that the two compositions 
$X_i\stackrel{f_i}\lrar Y_{\phi(i)}\rar Y_j$ and $X_i\stackrel{g_i}\lrar Y_{\psi(i)}\rar Y_j$  coincide. The map $i\mapsto j(i)$ thus induced is therefore a cofinal map $\Zee_+\rar\Zee_+$. It follows that $f$ and $g$ induce the same morphism on the objects $X_i\rar Y_{j(i)}$, hence $f$ and $g$ are equal as $S$-roofs, and $f=g$ in $\Funct$. Then, $H$ is faithful. 

\vspace{0.2cm}

(iii) Finally, $H$ is clearly essentially surjective, since each countable ind-object in $\C$ can be written as $\limXi$, with $I=\Zee_+$. The theorem is proved.
\end{proof}
Let now be $\phi$ a cofinal functor $\Zee_+\rar\Zee_+$. Then $\phi$ gives raise to a canonically defined co-cofinal functor $\phi^o: (\Zee_+)^{op}\rar (\Zee_+)^{op}$ and conversely, any such co-cofinal functor comes from a unique cofinal functor $\Zee_+\rar\Zee_+$.

\par

Let us consider the category $\Funcop$ of contravariant functors from $\Zee_+$ to $\C$, and the corresponding category $\Funcopt$. In this category, consider the following class of morphisms:
\begin{multline}
T:=\biggl\{Y\cdot\phi_0^{op}\lrar Y\cdot\phi_1^{op} \text{ such that } Y\in\Func; \ \phi_0, \phi_1: \Zee_+\rar\Zee_+ \text{ are cofinal, } \\ 
\text{ and }\phi_0\geq\phi_1 \biggr\}.
\end{multline}

Then we have:
\begin{prop}\label{tlocal}
The class $T$ is a localizing class of morphisms in $\Funcopt$.
\end{prop}
\begin{theorem}\label{prolocal}
$\Funcopt[T^{-1}]$ is equivalent to $\Pro_{\aleph_0}(\C)$.
\end{theorem} 
The proof of these statements are obtained from those of the analogous Proposition \eqref{Slocal} and Theorem \eqref{indlocal} with the obvious modifications. Notice that in this case, whenever we have an object $Y\in\Funcopt$ and a morphism $t\in T$ of the form $Y\rar Y\cdot\phi^{op}$, for some nondecreasing cofinal $\phi:\Zee_+\rar\Zee_+$, it is $\id_{\Zee_{+}}\geq\phi$.

\vspace{0.3cm}

Since the equivalence $(\Zee_+)^{op}\stackrel{\sim}\lrar\Zee_-$, the preorder of nonpositive integers, it is possible to describe equivalently $\Pro_{\aleph_0}(\C)$ as a localization of {\it covariant} functors $\Zee_-\rar \C$ in the following way: given the category of covariant functors $\Funcneg$, let be $\phi:\Zee_-\rar\Zee_-$ a cofinal functor with $\phi(0)=0$. As a map of ordered set $\phi$ is a monotonic nondecreasing function $\Zee_-\rar\Zee_-$ such that $\lim_{i\to-\infty}\phi=-\infty$. Then, consider the following class of morphisms in $\Funcnegt$:
\begin{multline}
T_-:=\biggl\{Y\cdot\phi_0\lrar Y\cdot\phi_1 \text{ such that } Y\in\Funcneg; \ \phi_0, \phi_1: \Zee_-\rar\Zee_- \text{ are cofinal, } \\ 
 \text{ and }\phi_0\leq\phi_1 \biggr\}
\end{multline}
In this setting, we can reformulate Proposition\eqref{tlocal} and Theorem\eqref{prolocal} by claiming that $T_-$ is a localizing class of morphisms in $\Funcnegt$ and that $\Funcnegt[T_-^{-1}]$ is equivalent to $\Pro_{\aleph_0}(\C)$. Notice that in this case, whenever we have an object $Y\in\Funcnegt$ and a morphism $t\in T_-$ of the form $Y\rar Y\cdot\phi$, for some nondecreasing cofinal $\phi:\Zee_-\rar\Zee_-$, it is $\id_{\Zee_{-}}\leq\phi$. This reformulation will be used in the next section.

\subsection{Ind/Pro iteration and localization}\label{indproitloc}
Let $I$ and $J$ be filtering countable preorders as above.
\begin{definition}
The full subategory of $\Ind_{\aleph_0}\Pro_{\aleph_0}(\C)$ whose objects are formal limits $\indprorigaa$, for bifunctors $X:I^{op}\times J\rar \C$, is called the category of {\it straightified (countable) ind/pro objects} on $\C$, and it is denoted by $\IP(\C)$. The full subcategory of $\IP(\C)$ whose objects are strict ind/pro limits is called the {\it strict} category of straightified ind/pro objects and denoted by $\IP^s(\C)$. The categories $\PrI(\C)$, $\PrI^s(\C)$ 
are defined in a similar way.
\end{definition}
Motivated by the construction of the category $\limA$ of Beilinson we prove a localization theorem also for this category, completely analogous to Theorem \eqref{indlocal}. 

\vspace{0.3cm}

Let be  $\Pi:=\{(i,j)\in\Zee\times\Zee \ | \ i\leq j\}$. Then $\Pi$ is naturally 
a preorder with the order induced from $\Zee\times \Zee$. Let be $\Bifunc$ the category of functors $\Pi\rar \C$.

\begin{definition}
The category $\IP_{\Pi}(\C)$ is the category of formal limits $\indprorigaa$ for functors $X\colon\Pi\to\C$.
\end{definition}
Let be $I$, $J$ two countable preorders. We shall call a set 
$\Sigma\subset I^o\times J$ {\it cofinally dense} if, for each $j$, the set 
$\Sigma_j=\Sigma\cap(I\times\{j\})$ is cofinal in $I$. \\
Then, we define the category $\widetilde{\IP}(\C)$ to be the category whose objects 
are formal limits $\limsigma$, for functors $X\colon\Sigma\to\C$. 

\vspace{0.1cm}

It is clear that $\Pi$ is cofinally dense in $\Zee^o\times\Zee$, so 
$\IP_{\Pi}(\C)$ is a particular case of the category $\widetilde{\IP}(\C)$.
\begin{prop}
The category $\IP(\C)$ embeds as a full subcategory of $\widetilde{\IP}(\C)$.
\end{prop}
\begin{proof}(Sketch)
There is a functor $F\colon\IP(\C)\to\widetilde{\IP}(\C)$, defined on the objects as 
$$
F(\limXij)=\limsigma
$$
and extended to the morphisms in a natural way. This functor induces a bijection on the Hom-sets. This is a consequence of the following, whose proof is clear:
\begin{lm}
Let $I$ be a countable preorder, and $H\hrar I$ be cofinal in $I$. Let be 
$A=\limproi A_i$ and $B=\limproi B_i$. If for all $h\in H$ there are maps 
$g_h\colon A_h\to B_h$ compatible with the structure morphisms of the two pro-systems $\{A_i\}$ and $\{B_i\}$, then it is possible to extend $\{g_h\}$ to a 
morphism of pro-objects $g\colon\limproi A_i\to\limproi B_i$.
\end{lm}
With a similar argument it is also proved that $F$ induces an injection on the Hom-sets.
\end{proof}
\begin{prop}\label{IPequivalent}
The category $\IP_{\Pi}(\C)$ is equivalent to the category $\IP(\C)$.
\end{prop}
\begin{proof}
It is enough to show that for $I=J=\Zee$ and $\Sigma=\Pi$, the functor $F$ 
of the previous proposition is essentially surjective. 

\vspace{0.1cm}

Thus, for each $j$ it is $\Pi_j=\{i\in\Zee \ |\ i\leq j\}$. Let be 
$X\colon\Pi\to\C$ a functor, and consider the ind/pro object 
$\limpi$. It is possible to extend such object to the complement of $\Pi$, 
as follows: consider first, for each $j$, the object 
$X_j={\underset{i\leq j}{``\varprojlim"}}X_{i,j}$. Extend $X_j$ to a pro-object 
$\widetilde{X}_j={\underset{i\in \Zee}{``\varprojlim"}}\widetilde{X}_{i,j}$, where:
$$
\widetilde{X}_{i,j} = \left\{ \begin{array}{rl}
 X_{i,j} &\mbox{ if $i\leq j$} \\
  X_{j,j} &\mbox{ if $i>j$}
       \end{array} \right.
$$
Next, consider $\limj\ \widetilde{X}_{i,j}$. The structure maps of this ind/pro 
system are those induced in a obvious way by $X_{i,j}$. Thus,  
$\limj\ \widetilde{X}_{i,j}$ is in $\IP(\C)$, and it is clear that 
$F(\limj\ \widetilde{X}_{i,j})\cong\limpi$. Thus $F$ is essentially surjective, 
and the proposition is proved.
\end{proof}
Let be $\phi: \Zee\rar\Zee$ a cofinal and co-cofinal functor. We shall call such functors {\it bicofinal}. As a map, $\phi$ is bicofinal if and only if $\phi$ is nondecreasing, with $\lim_{n \to -\infty}\phi = -\infty$ and $\lim_{n \to\infty}\phi = \infty$. We can associate to each bicofinal functor $\phi$ an endofunctor $\tilde\phi: \Pi\rar\Pi$, defined on the objects as 
\begin{equation}\label{tildephi}
\tilde\phi(i,j):=(\phi(i), \phi(j)).
\end{equation}
Since $\phi$ is nondecreasing,  $\tilde\phi$ is well defined and it is  cofinal since $\phi$ is bicofinal. 

\par

We write $\phi\leq\psi$ whenever, for all $i\in\Zee$, we have $\phi(i)\leq\psi(i)$. Then it is clear that if $\phi\leq\psi$ there is a natural transformation of functors $X\cdot\tilde\phi\rar X\cdot\tilde\psi$, defined in components by $X_{\phi(i),\phi(j)}\rar X_{\psi(i),\psi(j)}$.

\begin{prop}\label{Sslocal}
The class 
$$
U:=\biggl\{ X\cdot\tilde\phi_0\rar X\cdot\tilde\phi_1, \ \  \phi_0, \phi_1: \Zee\rar\Zee \text{ bicofinal and } \phi_0\leq\phi_1\biggr\}
$$ 
is a localizing system of morphims in $\Bifunc$.
\end{prop}

\begin{proof}
The proof of this proposition is just an adaptation of the proof of Proposition \eqref{Slocal} to the more general case of functors $X_{i,j}:\Pi\rar\C$. Notice that in this case it is not necessary to introduce the analog of the category 
$\Funct$. Given the function $\phi$ with $\id\leq\phi$, the corresponding function $\psi$ with $\psi\leq\id$ is defined in the same way as there, and 
the fact $\lim_{i\to -\infty}\phi(i)=-\infty$ assures that $\psi$ is well defined on the whole $\Zee$ and bicofinal. The claims (a), (b) are proved in essentially the same way. Regarding (c), if $f,g:X\rar Y$ are two natural transformations of the functors $X$ and $Y$, the only difference with the proof of (c) of Proposition \eqref{Slocal} is that we use the naturality of 
$f_{i,j}$ and $g_{i,j}$ in the diagrams corresponding to the diagrams \eqref{natural}.
\end{proof}

\begin{cor}\label{biroofs}
The localized category $\Bifunc[U^{-1}]$ is equivalent to the category of $U-$roofs over $\Bifunc$.
\end{cor}
We can describe an object of $\Bifunc[U^{-1}]$ as an object of $\Bifunc$, and a morphism $X\rar Y$ of $\Bifunc[U^{-1}]$ as an equivalence class of $U-$roofs, i.e. a pair $(f, \tilde\phi)$, where now $\tilde\phi$ is a cofinal map $\Pi\rar\Pi$ coming from a bicofinal map $\phi\colon\Zee\to\Zee$ and $f$ is a natural transformation $X\rar Y\cdot\tilde\phi$, which in components can be written, for all $(i, j)\in\Pi$ as $X_{i,j}\rar Y_{\phi(i), \phi(j)}$. Two roofs $(f, \tilde\phi)$ and $(g, \tilde\psi)$ are equivalent if and only if there exists a third roof $(h, \tilde\theta)$ forming a commutative diagram like Diagram \eqref{eqroof}. In particular, $(f, \tilde\phi)$ and $(g, \tilde\psi)$ are equivalent if and only if there exists a cofinal $\tilde\theta:\Pi\rar\Pi$ such that $f$ and $g$ induce the same morphism on the objects $X_{i,j}\rar Y_{\theta(i), \theta(j)}$. Notice that in this case we have 
$\tilde\theta(i,j)\geq \max(\phi(i), \psi(i)), \max(\phi(j), \psi(j))$.

\begin{theorem}\label{indprolocal}
Let be $U$ as in Prop. \eqref{Sslocal}. The localized category $\Bifunc[U^{-1}]$ is equivalent to $\IP(\C)$.
\end{theorem}
\begin{proof} From Proposition \eqref{IPequivalent} it will be enough to 
prove the existence of an equivalence of categories 
$\Bifunc[U^{-1}]\cong\IP_{\Pi}(\C)$. The proof of this theorem will then be, as for the case  Prop. \eqref{Sslocal},  an adaptation of the proof of the corresponding theorem \eqref{indlocal} for $\Funct$. Given an object $X=\{X_{i,j}\}_{i\leq j} \in\Bifunc$, we define an ind-pro object as follows:
$$
\Phi(X):= \indprorigaa.
 $$ 
The correspondence $\Phi$ is then extended to the morphisms of $\Bifunc$ in the obvious way. It is easy to see that $\Phi$ is a functor 
$\Bifunc\rar\IP(\C)$. For this functor, we can prove in a similar way as in Theorem \eqref{indlocal} that if $u\in U$, $\Phi(u)$ is an isomorphism. Then $\Phi$ gives rise to a canonically defined functor 
$\Bifunc[U^{-1}]\rar\IP(\C)$, that we shall call $\hat\Phi$. On the objects, we have $\hat\Phi(X)=\Phi(X)$ and on the morphisms, if $(f, \phi)=\{f_{i,j}: X_{i,j}\rar Y_{\phi(i), \phi(j)} \}_{(i,j)\in\Pi}$, then:
$$
\hat\Phi(f)=``\limj"{\underset{i\leq j}{``\varprojlim"}}f_{i,j}.
$$
We claim that this functor is an equivalence of categories. The fact 
that $\hat\Phi$ is full (i.e surjective on the Hom-sets) is proved still using straightification of morphisms in $\IP(\C)$. We now prove that $\hat\Phi$ is faithful (i.e. injective on the Hom-sets).

\vspace{0.2cm}

Let be $(f, \tilde\phi), (g, \tilde\psi): X\rar Y$ two morphisms in $\Bifunc[U^{-1}]$, such that $\hat\Phi(f)=\hat\Phi(g)$. For all $j$, define the objects of $\Pro_{\aleph_0}(\C)$:
\begin{align*}
X_j:& ={\underset{i\leq j}{``\varprojlim"}}X_{i,j};& Y_{\phi(j)}:& ={\underset{i\leq j}{``\varprojlim"}}Y_{\phi(i),\phi(j)};& Y_{\psi(j)}:& ={\underset{i\leq j}{``\varprojlim"}}Y_{\psi(i),\psi(j)}. \\
\end{align*}
and the morphisms
\begin{align*}
f_j:& ={\underset{i\leq j}{``\varprojlim"}}f_{i,j}:X_j\lrar Y_{\phi(j)};& g_j:& ={\underset{i\leq j}{``\varprojlim"}}g_{i,j}:X_j\lrar Y_{\psi(j)}. \\
\end{align*}
Then, we have $f=``\limj"f_j=``\limj"g_j=g$, an equality of ind-morphisms. Then, for all $j$ there exists a $j'=\delta(j)$, such that the following diagram is commutative in $\Pro(\C)$:
$$
\xymatrix{ 
& Y_{\phi(j)} \ar[rd] &&\\
X_{j} \ar[ru]^{f_j} \ar[rd]_{g_{j}} & &Y_{\max(\phi(j), \psi(j))}\ar[r] &Y_{\delta(j)} \\
& Y_{\psi(j)} \ar[ru] &&
}
$$
and the map $j\mapsto\delta(j)$ is nondecreasing, with $\delta(j)\geq\max(\phi(j), \psi(j))$.

\vspace{0.2cm}

On the other hand, let's write  $Y_{\delta(j)}={\underset{i'\leq\delta(j)}{``\varprojlim"}}Y_{i',\delta(j)}$.  Given $j$, we express the equality in $\Pro(\C)$ of the morphisms in the above diagrams as follows: for all $i'$, with $i'\leq\delta(j)$, there exist $i_0=i_0(i')$ and $i_1=i_1(i')$, and an $i\leq\min(i_0, i_1)$, such that the diagram 
\begin{equation}\label{proequal}
\xymatrix{ 
&& X_{i_0, j} \ar[rd] &\\
X_{i, j}\ar[r]& X_{\min(i_0, i_1), j} \ar[ru] \ar[rd] & &Y_{i', \delta(j)} \\
&& X_{i_1, j} \ar[ru] &
}
\end{equation}
commutes. In particular, since 
$\lim_{i'\to -\infty}i_0(i')=\lim_{i'\to -\infty}i_1(i')=-\infty$, the map 
$i'\mapsto i(i')$ has limit $=-\infty$ as $i'\to -\infty$.

Then, the two maps $i'\mapsto i(i')$ and $j\mapsto\delta(j)$ combine 
together to a unique map $k\to\theta(k)$, which is a bicofinal, nondecreasing 
map $\Zee\to\Zee$, for which we can rewrite diagram \eqref{proequal} as 
$$
\xymatrix{ 
& X_{i_0, j} \ar[rd] &\\
X_{i, j} \ar[ru] \ar[rd] & &Y_{\theta(i), \theta(j)} \\
& X_{i_1, j} \ar[ru] &
}
$$
which expresses the equality of $(f, \tilde\phi)$ and $(g, \tilde\psi)$ as $U$-roofs. Then, $\hat\Phi$ is faithful. Finally, the same argument used for Theorem \eqref{indlocal} applies also to prove that $\hat\Phi$ is essentially surjective. Then $\hat\Phi$ is an equivalence and the theorem is proved.
\end{proof}
In a similar way, we can prove that also  $\PrI(\C)$ can be obtained as a localized category as follows.

\vspace{0.2cm}

 Define the set $\Lambda:=\{(i,j)\in\Zee\times\Zee \ | \ i\geq j\}$. Then $\Lambda$ is naturally an ordered set  with the order induced from $\Zee\times \Zee$. Let be $\Bifuncl$ the category of functors $\Lambda\rar \C$. Then we have the following statements, whose proofs are the analog of the proofs of Propositions \eqref{Sslocal} and Theorem\eqref{indprolocal}:
 
 \begin{prop}\label{Vlocal}
 The class $V:=\{ X\cdot\tilde\phi_0\rar X\cdot\tilde\phi_1, \ \  \phi_0, \phi_1: \Zee\rar\Zee \text { bicofinal and } \phi_0\leq\phi_1\}$ is a localizing system of morphims in $\Bifuncl$.
\end{prop}

\begin{theorem}\label{proindlocal}
Let be $V$ as in Prop. \eqref{Vlocal}. The localized category $\Bifuncl[V^{-1}]$ is equivalent to the cateogry $\PrI(\C)$.
\end{theorem}
 
\section{Locally compact objects and the Kato category} 
In this section, let be $\mathcal P$ the category of compact Hausdorff totally disconnected spaces and $\mathcal L$ the category of locally compact Hausdorff totally disconnected spaces.
\begin{theorem}\label{loc} 
(1) The category $\Pro(\fset)$ is equivalent to the category $\Pro^s(\fset)$. 
\\
(2) The category $\Pro^s(\fset)$ is equivalent to $\mathcal P$, via the functor that assigns to each pro-finite set its projective limit. \\
(3) \text{(Kapranov)} The category $\mathcal L$ can be identified with the full subcategory of $\Ind^s\Pro^s(\fset)$, whose objects can be represented as ind-pro systems $X\colon I^{op}\times J\to\fset$, where I and J are filtering categories, and whose squares are cartesian.
\end{theorem}

\begin{proof} (1) and (2) are elementary. We sketch here a proof of (3), taken from \cite{ka2}. 

\vspace{0.1cm}

Consider $\limXij$. Let $X_j=\limproi X_{i,j}$, so that $X=``\limj"X_j$. 

Every square of the system $X_{ij}$,
\begin{equation}\label{cartesian}
\xymatrix{ 
X_{i'j}\ar[d]
\ar[r]
& X_{i'j'} \ar[d] \\
X_{ij}\ar[r]_{}
& X_{ij'}
}
\end{equation}
is cartesian, the horizontal maps are injections and the vertical ones are surjections.  Because of this, the induced map $X_j\rar X_{j'}$ is an open embedding, so $X$ is locally compact and it is easy to see that $X$ is totally disconnected.

\vspace{0.2cm}

Conversely, suppose that $X$ is a locally compact totally disconnected space. We show that there exists a representation of it as the limit 
an ind-pro system \proindrigaa \ for an ind/pro system whose squares are cartesian. 
\begin{lm}
Let be $X$ as above. Then $X$ can be expressed as a filtering direct limit 
$X=\underset{j}\varinjlim\ X_j$, where each $X_j$ is both open and closed profinite space, and whose structure maps are embeddings.
\end{lm}
\begin{proof}
Suppose $X$ is a locally compact, totally disconnected Hausdorff space. 
Then, for all $x\in X$ there exists an open set $U$, such that its closure $\bar U$ is a profinite space. Thus, for all $x$ there is a neighborhood $V$ 
of $x$ which is both open and closed. Thus, $X$ can be written as the union 
of such $V$'s, i.e. $X=\underset{V\subset X}\varinjlim V$.
\end{proof}
Next, for every profinite space $X$, Let be 
$$
J_X=\{\mathcal R | \mathcal R \ \text{is an equivalence relation on X, is finite, and whose classes are open in\ } X\}.
$$ 
It is known (see e.g. \cite{rz}), that $X=\varprojlim_{\mathcal R\in J_X}X/\mathcal R$.

\vspace{0.1cm}

Furthermore, for an open embedding $Y\lrar X$ of profinite spaces,  if we write  $Y=\varprojlim_{\mathcal R'\in J_Y}Y/\mathcal R'$, then $J_Y$ is a cofinal subset of $J_X$. Moreover, for every $\mathcal R' \subset J_Y$ 
and $\mathcal R \subset J_X$ such that $\mathcal R' \subset \mathcal R$, the square
$$
\xymatrix{ 
Y/{\mathcal R}\ar[d]
\ar[r]
& Y/{\mathcal R'} \ar[d] \\
X/{\mathcal R} \ar[r]
& X/{\mathcal R'}
}$$
is cartesian.
\end{proof}
The previous theorem gives some motivation for the following
\begin{definition} Given a category $\mathcal C$, let be $\mathcal{L(C)}$ the full subcategory of $\IPs(\C)$ whose squares are cartesian. The category $\mathcal{L(C)}$ is called the {\it category of locally compact objects} of $\mathcal C$ of countable type.
\end{definition}
\begin{definition}\label{kato} 
{\bf The Kato category.}  The full subcategory of $\mathcal{L(C)}$, whose 
squares are also cocartesian, is called {\it the Kato category} associated with $\C$, and denoted $\kappa(\C)$.
\end{definition}

\begin{remark} If $\C$ is an abelian category, we have $\mathcal L(C)=\kappa(\C)$. Indeed, in an abelian category a commutative square 
$$
\xymatrix{ 
a\ar[d]_{g'}
\ar[r]^{f'}
& c\ar[d]^{g} \\
b\ar[r]_{f}
& d
}
$$
where $f, f'$ are monomorphisms and $g, g'$ are epimorphisms, is cartesian if and only if it is cocartesian. In general, however, $\kappa(\C)\neq\mathcal L(\C)$.
\end{remark}

The  following is stated in \cite{k}, and proved in \cite{kv}.

\begin{prop} $\kappa(C)$ embeds fully and faithfully in both $\IP^s(\C)$ and $\PrI^s(\C)$.
\end{prop}
Let be $\C$ any category. We refer to the definitions and the propositions stated in sect. \eqref{indproitloc}.

\begin{definition}
$\prelim(\C)$ is the full subcategory of $\Bifunc[U^{-1}]$, whose objects are functors $X_{i,j}:\Pi\rar \C$, such that for all $i\leq i'$ and $j\leq j'$, we have: 
$$
X_{i',j}\lrar X_{i, j} \text{ is an epimorphism,}
$$
$$
X_{i,j}\lrar X_{i,j'} \text{ is a monomorphism, }
$$
and
$$
\xymatrix{ 
X_{i',j}\ar[d]
\ar[r]
&  X_{i, j}\ar[d] \\
X_{i', j'}\ar[r]_{}
& X_{i, j'}
}
$$
is cartesian.

\end{definition}

\begin{prop} There exists an equivalence of categories $\Psi:\prelim(\C)\lrar\mathcal L(\C)$, for which the diagram
$$
\xymatrix{ 
\prelim(\C)\ar[d]
\ar[r]^-{\Psi}
& \mathcal L(\C) \ar[d] \\
\Bifunc[U^{-1}]\ar[r]^-{\hat\Phi}
& \IP^s(\C)
}
$$
where the vertical arrows are embeddings, is commutative.
\end{prop}

\begin{proof} 
Consider the embedding $\prelim(\C)\hrar\Bifunc[U^{-1}]$. Define $\Psi$ to be the restriction of the equivalence $\hat\Phi$ to the full subcategory $\prelim(\C)$. Then it is clear that $\Psi$ is a functor $\prelim(\C)\rar\mathcal L(\C)$ and that it is an equivalence. The proposition follows.
\end{proof}

\section{Exact categories}

\subsection{Exact categories and their abelian envelopes}

\begin{definition} [Quillen] 
An {\it exact category} is a pair $(\mathcal A, \mathcal E)$, where $\mathcal A$ is an additive category and $\mathcal E$ is a class of sequences of the type 
\begin{equation}\label{exac}
0\xrar{} a'\xrar{i} a\xrar{j} a''\xrar{} 0
\end{equation}
called the {\it admissible short exact sequences} of $\A$. A morphism which occurs to be the map $i$ of some member of the family $\E$ will be called an {\it admissible monomorphism}, while a morphism which occurs to be the map $j$ of some member of the family $\E$ will be called an {\it admissible epimorphism}. We require that  the following axioms are satisfied: \\
(1) If a sequence is isomorphic to a sequence in $\mathcal E$, then the sequence is in $\mathcal E$. \\
(2) For any pair of objects $a', a''$ of $\mathcal A$, the following short exact sequence is in $\E$:
$$
0\xrar{} a'\xrar{(id, 0)} a' \oplus a''\xrar{pr_2} a''\xrar{} 0
$$
(3a) If $a'\lrar a''$ is an admissible epimorphism, then, for every arrow $f: b''\lrar a''$ of $\A$, there exists a cartesian square
$$
\xymatrix{ 
b'\ar[d]
\ar[r]^{j'}
& b'' \ar[d]^{f} \\
a'\ar[r]_{j}
& a''
}$$
such that the arrow $j'$ is an admissible epimorphism. \\
(3b) Dually,  If $a\lrar a'$ is an admissible monomorphism, then, for every arrow $f: a\lrar b$ of $\A$, there exists a cocartesian square
$$
\xymatrix{ 
a\ar[d]_{f}
\ar[r]^{i}
& a' \ar[d] \\
b\ar[r]_{i'}
& b'
}$$
such that the arrow $i'$ is an admissible monomorphism. \\
(4) Let $f: b\rar c$ an arrow whose kernel is in $\A$. If there exists an arrow $a\rar b$ such that the composition $a\rar b\rar c$ is an admissible epimorphism, then $f$ is an admissible epimorphism. Dually for admissible monomorphisms.
\end{definition}
\begin{examplesremarks}
(a) For a shorter system of axioms,  see \cite{ke}. \\ 
(b) Any additive category can be made an exact category in a canonical way, by taking $\E$ to be the set of the split exact sequences. In particular, there are two canonical ways to turn an abelian category into an exact category: by taking $\E$ to be either the class of split sequences, or the class of all the short exact sequences in the category. When one refers to an abelian category as an exact category, it is usually meant the latter way. \\
(c) If $(\A, \E)$ is an exact category, its dual category $\A^{op}$ is also exact in a natural way, since the defining axioms for $\E$ are self-dual. In particular, an admissible monomorphism of $\A^{op}$ is the opposite of an admissible epimorphism of $\A$, and an admissible epimorphism of $\A^{op}$ is the opposite of an admissible monomorphism of $\A$.
\end{examplesremarks}
If $(\A, \E_A)$ and  $(\B, \E_B)$ are two exact categories, an {\em exact functor} $F: \A\rar \B$ is an additive functor taking admissible short exact sequences into admissible short exact sequences.
\begin{definition}
(1) Given a full subcategory $\B$ of an abelian category $\mathcal F$, we shall say that 
$\B$ is {\it closed under extensions} whenever for every short exact sequence $0\rar b'\hrar x\epi b''\rar 0$ of $\F$, with $b, b''\in\B$, we have $x\in\B$. \\
(2) A {\it fully exact subcategory} of an exact category $(\A, \E_{\A})$ is a full additive subcategory $\B\subset \A$ which is closed under extensions.
\end{definition}

If this condition is satisfied, then it is possible to endow $\B$ with a structure of an exact category, by defining the family $\E_{\B}$ of admissible short exact sequences as those sequences of $\E_{\A}$ whose terms are in $\B$. In this way, the inclusion functor  $\B\subset \A$ becomes a fully faithful exact functor.

\par

We have the  following important

\begin{theorem}\label{embedding} \cite{q} Let $(\A, \E)$ be an exact category. Let $\mathcal F$ be the additive category of the contravariant functors from $\A$ to the abelian category of abelian groups which are left exact, i.e. those functors that carry short exact sequences of $\A$ of the form \eqref{exact} into exact sequences $0\rar Fa''\rar Fa\rar Fa'$. Then, $\mathcal F$ is an abelian category, and the Yoneda functor h embeds $\A$ as a full subcategory of $\mathcal F$, closed under extensions, in such a way that a short exact sequence is in $\E$ if and only if h carries it into an exact sequence of $\mathcal F$.
\end{theorem}
The category $\F$ of the Theorem is called the {\it abelian envelope} of the exact category $(\A, \E)$. We shall call the embedding 
$h$ of the theorem the {\it Quillen embedding}.

\begin{prop}\label{yoneda} The Quillen embedding $h:\A\hrar \F$ is additive and left exact.
\end{prop}

For the proof of the proposition, see \cite{we}.

\vspace{0.2cm}

An immediate consequence of theorem \eqref{embedding} is the following characterization of an exact category:

\begin{cor}\label{extensionclosed} A category $(\A, \E)$ is exact if and only if $\A$ is a full subcategory,  closed under extensions, of an abelian category
$\F$. In such a case $\E$ is the class of short exact sequences of $\F$ whose terms are in $\A$.
\end{cor} 

The following will be useful:

\begin{lm}\label{cokernels} 
Let $m:a\hrar b$ be a monomorphism of $\F$ with $a\in\A$. Then, the following are equivalent: \\
(1) $m$ is admissible; \\
(2) $\coker(m)$ is in $\A$. \\
Dually for epimorphisms $e: c\epi a$.
\end{lm}

\begin{proof} 

\par

(1)$\Rightarrow$(2) is trivial. \\
(2)$\Rightarrow$(1):  If $m$ has a cokernel in $\A$, then we obtain a short exact sequence $a\hrar b\epi c$, with $a$ and $c$ in $\A$.  Being $\A$ closed under extensions, $b$ is in $\A$. This sequence is then a short exact sequence of $\F$ made by objects of $\A$, therefore by Theorem \eqref{embedding} is an admissible short exact sequence of $\A$. Then $m$ is an admissible monomorphism.
\end{proof}

\begin{definition} Let be $\A$ an exact category.  A commutative square
$$
\xymatrix{ 
A\ar[d]
\ar[r]
& B\ar[d] \\
C\ar[r]
& D
}$$
in which the horizontal arrows are admissible monomorphisms, and the vertical arrows are admissible epimorphisms, will be called an {\it admissible square}.
\end{definition}

\begin{prop}\label{cartesian}  Let $(\A, \E)$ be an exact category. Let be 
\begin{equation}\label{square}
\xymatrix{ 
a\ar[d]_{g'}
\ar[r]^{f'}
& d\ar[d]^{g} \\
b\ar[r]_{f}
& c
}
\end{equation}
a commutative square in $\A$.

Then: \\
(1) The square is cartesian in $\A$ $\Leftrightarrow$ it is cartesian in the abelian envelope $\F$ of $\A$. \\ 
(2) Suppose that the square is admissible. Then it is cartesian in $\A$ if and only if it is cocartesian in $\A$.
\end{prop}

\begin{proof}
We start with a general
\begin{lm} 
In an additive category $\A$, the pair of arrows $b\xrar{f} c\xlar{g} d$ has a pullback $b\xlar{g'} a\xrar{f'} d$ if and only if $a$ is the kernel object of the arrow $b\oplus d\xrar{fp_1-gp_2} c$, where $p_1$ and $p_2$ are the projections of the biproduct $b\oplus d$ resp. on b and on d, and, having set $m=ker (fp_1-gp_2)$, it is $f'=p_2m$, $g'=p_1m$.
\end{lm}
The proof of the lemma is a straightforward application of the universal property of the pullback (for the ``if" clause), and of the kernel (for the ``only if" part of the statement). See \cite{ml}.
We now apply the lemma and the theorem to prove the proposition. 

\vspace{0.2cm}

(1) ($\Rightarrow$) Consider the cartesian square \eqref{square}.  From the previous lemma,  the pullback $a$ can be described as a left exact sequence
$$
0\lrar a\xrar{m} b\oplus d\xrar{fp_1-gp_2} c.
$$
Apply the Quillen embedding $h$. From Corollary \eqref{yoneda}, being $h$ left exact and additive, we obtain a left exact sequence
$$
0\lrar h(a)\xrar{h(m)} h(b)\oplus h(d)\xrar{h(f)p_1-h(g)p_2} h(c),
$$
(where we have denoted still by $p_1$ and $p_2$ the projections), and which says that $h(a)$ is the kernel object of the corresponding arrow of $\F$, $h(f)p_1-h(g)p_2$. This means that $h(a)$ is the pullback object of $h(f)$ and $h(g)$ in $\F$, i.e. $a$ remains the pullback of $f$ and $g$ also in $\F$.

\vspace{0.1cm}

($\Leftarrow$)  is trivial.

\par

Notice that dualizing this part of the proposition we obtain that a square is cocartesian in $\A$ if and only if is cocartesian in $\F$.

\vspace{0.2cm}

(2) Suppose  now \eqref{square} is an admissible cartesian square in $\A$. From (1) it follows that \eqref{square} is cartesian in $\F$. But $\F$ is abelian, and in an abelian category an admissible square is cartesian if and only if is cocartesian. Thus, the square is cocartesian in $\F$. Now apply (1)  again, and obtain that it is cocartesian in $\A$. The same argument, with ``cartesian" and ``cocartesian" exchanged, proves the converse of the implication. 
\end{proof}

\subsection{Ind/Pro objects and Exact categories} 

Let $\A$ be an exact category. Our next  goal is to prove that if $\A$ is an exact category, the categories $\Ind(\A)$, $\Pro(\A)$ inherit an exact structure from  that of $\A$, and to determine the exact structure of each category. For this aim, it will be useful to recall   the behavior of $\Ind(\C)$, $\Pro(\C)$ with respect to finite limits of the base category $\C$.

\begin{prop}\label{indlimits}
Let be $J$ a small filtering category and $\C$ any category. Suppose that $\C$ has finite inductive (projective) limits. Then the functor
$$
Hom(J^{op}, \C)\lrar \Pro(\C)
$$
defined by 
$$
(Y_j)_{j\in J}\longmapsto \proYj
$$ 
commutes with finite inductive (projective) limits. Similarly for $\Ind(\C)$.
\end{prop}

\begin{proof} (see \cite{am}).
Let be $F=\{F_j\}_{j\in J}$ a finite diagram of functors $J^{op}\rar\C$. Let us denote by $Y_{\alpha}=\{Y_{\alpha j}\}$ the objects which compose the diagram $F$. Let's take $\varinjlim {F}=\{\varinjlim {F_j}\}$ and let us denote this object by $\hat{Y}=\{Y_j\}$. Then:
\begin{equation}\label{limits}
\begin{split}
Hom(\hat{Y}, Z)& =\limcom \\
&=\limcoma \\
&=\limcomak \\
&=\limcomjka \\
&=\underset{\alpha}\varprojlim Hom_{\C}(Y_{\alpha}, Z).
\end{split}
\end{equation}
Where the first equality is simply the definition of morphism in $\Pro(\C)$, eq. \eqref{prohom}. The third equality is because the functor $\underset{j}\varinjlim$ commutes with finite inverse limits (see \cite{sga}), and the fourth because $\underset{k}\varprojlim$ commutes with inverse limits. Then, equation \eqref{limits} shows that the object $\hat{Y}$ is a direct limit in $\Pro(\C)$. Similarly for the proof for an inverse limit, and for the case of $\Ind(\C)$.
\end{proof}

\begin{prop}\label{indcloslimits}
If $\C$ is closed under finite inductive (projective) limits, then also $\Ind(\C)$, $\Pro(\C)$ are closed under finite inductive (projective) limits.
\end{prop}

\begin{proof} Let us prove the claim again for $\Pro(\C)$. To show that $\Pro(\C)$ is closed under arbitrary finite limits, it is enough to prove that $\Pro(\C)$ is closed under coproducts and pushouts (see \cite{sga}). To see that $\Pro(\C)$ has coproducts, for instance, we start with a diagram in $\Pro(\C)$: 
$$
\xymatrix{
Z & X  \ar[l]\ar[r] & Y \\
}
$$
From Lemma \eqref{straight}, this diagram can be straightified to a system of diagrams of $\C$:
$$
\xymatrix{
Z_j & X_j  \ar[l]\ar[r] & Y_j \\
}
$$
 For $j\in J$. Then since in $\C$ pushouts exist, we construct them pointwise, i.e. for each $j$ in the above diagram. Now apply Proposition \eqref{indlimits}, and the claim is proved. Similarly for the case of inverse limits, and for $\Ind(\C)$.

\end{proof}

\begin{prop}\label{indproabelian}
If $\F$ is abelian, then $\Ind(\F)$, $\Pro(\F)$ are abelian.
\end{prop}

\begin{proof}
See \cite{ks}, Theorem (8.6.5).
\end{proof}

As a further consequence, we obtain the straightification of monomorphisms and epimorphisms in $\Ind(\F)$, $\Pro(\F)$.

\begin{prop}\label{monostr} (Straightification of monomorphisms and epimorphisms in $\Ind(\F)$, $\Pro(\F)$.) 
Let be $\F$ an abelian category. A monomorphism $m:X\hrar Y$ in the category $\Ind(\F)$ can be represented as in Lemma \eqref{straight} by a system of monomorphisms of $\F$: $\{m_i:X_i\hrar Y_i\}$. Similarly for $\Pro(\F)$.
\end{prop}

\begin{proof}
In fact, in an abelian category a monomorphism is always the kernel of a morphism, and an epimorphism a cokernel of a morphism of the category. But since ``kernel" and ``cokernel" are finite limits, Proposition\eqref{indcloslimits} applies to the abelian categories $\Ind(\F)$ and $\Pro(\F)$, and the claim follows. 
\end{proof}

\begin{prop}
Let be $(\A, \E)$ be an exact category. Then $\Ind(\A)$, $\Pro(\A)$ have natural structures of exact categories. 
\end{prop}

\begin{proof} 
Let us consider the abelian envelope $\F$ of $(\A, \E)$ (see Theorem \eqref{embedding}).  It is clear that we have an embedding of $\Ind(\A)$ in the category $\Ind(\F)$, which is abelian by Proposition \eqref{indproabelian}.  We  show that $\Ind(\A)$ is closed under extensions in $\Ind(\F)$. This will give at once an exact structure.

\par

Thus, let be $X\stackrel{m}\hrar Y\stackrel{e}\epi Z$ a short exact sequence of $\Ind(\F)$, where $X$ and $Z$ are in $\Ind(\A)$. We shall prove that $Y\in\Ind(\A)$. Lemma \eqref{monostr}  allows us to straightify the above short exact sequence. We thus obtain, in components, the following diagram:
$$
\xymatrix{ 
... \ar[r] & X_{i-1} \ar[d]_{m_{i-1}}\ar[r] &X_{i}\ar[d]_{m_i} \ar[r] &X_{i+1} \ar[d]_{m_{i+1}}\ar[r]&... \\
... \ar[r] & Y_{i-1} \ar[r]\ar[d]_{e_{i-1}} &Y_i \ar[r]\ar[d]_{e_i} &Y_{i+1} \ar[r]\ar[d]_{e_{i-1}}&... \\
... \ar[r] & Z_{i-1} \ar[r] &Z_{i} \ar[r] &Z_{i+1} \ar[r]&... \\
}
$$
In this diagram, for each index $i$ each column is a short exact sequence of $\F$, with $X_i, Z_i$ in $\A$. Since $\A$ is closed under extensions in $\F$ it follows that $Y_i\in\A$ for all $i$. Thus, $Y=\limYi$ is in $\Ind(\A)$, which is thus also closed under extensions in $\Ind(\F)$, and the statement follows. The same argument works for $\Pro(\A)$. 
\end{proof}

As a corollary, we have the structure of admissible short exact sequences of $\Ind(\A), \Pro(\A)$ and the straightification of admissible mono/epimorphisms of these categories:

\begin{cor}\label{straight2}
Let $m: X\hrar Y$ be an admissible monomorphism in the category $\Ind(\A)$. Then, it is possible to represent $X=\limXi$, $Y=\limYi$ through the same filtering category $I$ in such a way that $m$  can be represented by a system of admissible monomorphisms of $\A$: $\{m_i:X_i\hrar Y_i\}$, as in Lemma \eqref{straight}. Similarly for the admissible monomorphisms of $\Pro^a(\A)$ and for admissible epimorphisms.
\end{cor}

Also, we get:

\begin{cor} All the iterated ind/pro-categories obtained from  $\A$ are exact. 
\end{cor}

In particular, $\Ind\Pro(\A)$ and $\Pro\Ind(\A)$ are exact categories.
\par

\section {The Beilinson category}

We introduce here, following Beilinson, an alternative definition of locally 
 compact objects in a category $\A$, provided that $\A$ has the structure of an exact category. The resulting category of locally compact objects in $\A$ will be called the Beilinson category associated to $\A$ and denoted $\limA$. We also study its relation with the Kato category, $\kappa (\C)$, when $\C$ is exact.

\subsection{The Beilinson category of an exact category $(\A, \E)$}
We review, from our perspective of iterated ind/pro-objects, the construction of the category $\limA$ of Beilinson.
\par
Let $(\A, \E)$ be an exact category.  We refer to sect. \eqref{indproitloc} for the definition of the preorder  $\Pi$ and the terminology there introduced.

\begin{definition} 
Let be $X$ an object of $\AP$. We say that $X$ is {\it admissible}, whenever for any $i\leq j\leq k$, the corresponding sequence $X_{ij}\rar X_{ik}\rar X_{jk}$ is in $\E$. \ We will denote by $\APa$ \ the full subcategory of $\AP$ of the admissible objects.
\end{definition}

We also define $\EP$ as the class of sequences of objects of $\AP$, $\{ X\rar Y\rar Z\}$, such that, for all $i\leq j$ the induced sequences $X_{i,j}\rar Y_{i,j}\rar Z_{i,j}$ are admissible short exact sequences of $\A$.

\begin{lm}\label{lemma} 
The category $(\AP, \EP)$ is an exact category.
\end{lm}

\begin{proof} The proof of this lemma it is essentially an adaptation of the proof, in \cite{w} of the exactness of the categories $S_n(\A)$, to which we refer the reader.
\end{proof} 

\begin{lm}
The category $\APa$  has an induced structure of exact subcategory of $\AP$,  and there is an embedding $\A\hrar \APa$.  
\end{lm}

\begin{proof}
It is clear that if we have an admissible short exact sequence of $\EP$, whose end terms are admissible, also the middle term is admissible. Therefore, $\APa$ is closed under extensions in $\AP$, and so it is a (fully) exact subcategory. The embedding of the claim is described on the objects by sending every object of $\A$ into the one having $X_{i,-1}=X_{1,j}=0.$ 
\end{proof}

Notice that the category $\APa$ can be seen as a limiting case  of the Waldhausen categories $S_n(A)$ defined in \cite{w}, when $n$ ``goes to infinity". 
\vspace{0.3cm}

Let be $X\in\APa$, $\phi:\Zee\rar\Zee$ a bicofinal functor, and $\tilde\phi$ the induced map $\Pi\rar\Pi$ of Equation \eqref{tildephi}. Then it is clear that the object $X\cdot\tilde\phi$ is in $\APa$. We can thus define the class of morphisms:
$$
U_a:=\{ X\cdot\tilde\phi_0\rar X\cdot\tilde\phi_1, \text{ such that } X\in\APa, \  \phi_0, \phi_1: \Zee\rar\Zee \text { bicofinal and } \phi_0\leq\phi_1\}.
$$
With the same proof of Prop.\eqref{Sslocal} the following is proved

\begin{prop}
$U_a$ is a localizing class of morphisms in $\APa$.
\end{prop}

\begin{definition}\label{limA} The {\it Beilinson category} of an exact category $(\A, \E)$, is the category $\limA:=\APa[U_a^{-1}]$.
\end{definition}

We shall call the objects of the category $\limA$ also {\it generalized Tate 
spaces} relative to the exact category $\A$.
\begin{definition}  The category $\Ind^a(\A)$ is the full subcategory of $\Ind(\A)$ whose objects have structure morphisms which are admissible monomorphism. Similarly, $\Pro^a(\A)$ is the full subcategory of $\Pro(\A)$ whose objects have structure morphisms which are admissible epimorphisms. We call $\Ind^a(\A)$ and $\Pro^a(\A)$, respectively, the categories of {\it strictly admissible ind-objects} and the {\it strictly admissible pro-objects} of $\A$.
\end{definition}

\begin{definition}
The category $\IP^a(\A)$, of the {\it strictly admissible straightified ind/pro 
objects over $\A$ of countable type} is the full subcategory of $\IP(\A)$, whose objects are 
formal limits $\limXij$ for bifunctors $X:I^o\times J\rar\A$, such that for 
each $i\leq i'$ and all $j$ the morphism $X_{i',j}\rar X_{i,j}$ is an admissible 
epimorphism, and for each $j\leq j'$ and all $i$, the morphism 
$X_{i,j}\rar X_{i,j'}$ is an admissible monomorphism.
\end{definition}

Then, with the same proof used to prove Theorem \eqref{indprolocal} we prove the 

\begin{theorem}\label{embedlima}
There is an embedding of $\limA$ as a full subcategory of $\IP^a(\A)$, given on the objects by 
\[
X=(X_{i,j})\mapsto \limXij.
\]
\end{theorem}

\subsection{Comparison between the constructions of Beilinson and Kato}

Let $(\A,\E)$ be an exact category.  Our goal in this section is to introduce an ``admissible version" of the Kato category and to prove that it
provides an alternate description of $\limA$. In fact, our considerations will 
also make $\limA$  appear as the ``admissible version" of $\prelim(\A)$.

\begin{definition}\label{admkato}
The full subcategory of $\IP^a(\A)$, whose  squares are cartesian (hence cocartesian), will be called the {\it admissible Kato category} of the exact category $(\A, \E)$, and denoted $\kappa^a(\A)$.
\end{definition}
Thus, $\kappa^a(\A)$ is a full subcategory of $\IP^a(\A)$. 

\begin{lm}\label{objlima}
Let be $X\in Ob\ \limA$ and $i+1\leq j$. Then, the square
$$
\xymatrix{ 
X_{i,j}\ar[d]_{e_1}
\ar[r]^{m_1}
& X_{i, j+1}\ar[d]^{e_2} \\
X_{i+1, j}\ar[r]_{m_2}
& X_{i+1, j+1}
}$$
is admissible, cartesian and cocartesian.
\end{lm}
\begin{proof} First, from the fact that $X$ is an admissible object, when $i\leq j\leq j+1$, the sequence $X_{i,j}\hrar X_{i, j+1}\epi X_{j,j+1}$ is an admissible short exact sequence. Hence $m_1$ (and similarly $m_2$) is an admissible monomorphism, while the consideration of the admissible short exact sequence
$X_{i,i+1}\hrar X_{i, j}\epi X_{i+1, j}$ corresponding to $i\leq i+1\leq j$ proves that $e_1$ and $e_2$ are admissible epimorphisms.
\par
We now show that the above square is cartesian and cocartesian. From Proposition \eqref{cartesian}, it will suffice to show that the square is cartesian in the abelian envelope.  For this aim, we complete the square above as follows:
$$
\xymatrix{ 
X_{i,j}\ar[dd]_{e_1}
\ar[r]^{m_1}
& X_{i, j+1}
\ar[rd]^{e'_{1}}
\ar[dd]^{e_2} \\
&& X_{j, j+1}
\\
X_{i+1, j}\ar[r]_{m_2}
& X_{i+1, j+1}
\ar[ru]_{e'_2} \\
}$$
In this diagram now both $(m_1, e'_1)$ and $(m_2, e'_2)$ are thus admissible short exact sequences. Further, we use the standard diagram-chase technique in abelian categories; in particular, we speak about ``elements" of 
objects. See \cite{ml}, Theorem VIII.4.3.
\par
Thus, let be $a$ and $b$  be elements of, respectively, $X_{i+1, j}$ and $X_{i, j+1}$, such that $m_2(a)=e_2(b)$. Since $e_1$ is surjective, the statement is proved if we can produce a (necessarily unique) preimage of $b$ in $X_{i,j}$. That is, if and only if $b\in\im(m_1)=\ker(e'_1)$. Thus, everything reduces to prove that $e'_1(b)=0$. But 
$e'_1=e'_2e_2$, so let's consider $e_2(b)$. It is $e_2(b)=m_2(a)$, hence $e_2(b)\in\im(m_2)=\ker(e'_2)$. It follows $0=e'_2e_2(b)=e'_1(b)$, as required. 
\par
Thus, the square is cartesian. Part 2) of proposition \eqref{cartesian} shows now that the square is also cocartesian.
\end{proof}

\begin{prop} There is an embedding of $\limA$ as a full subcategory of 
$\kappa^a(\A)$. 
\end{prop}

\begin{proof}
Lemma \eqref{objlima} shows the existence of a functor 
$i\colon\limA\to\kappa^a(\A)$, given on the objects by $i(X)=\limXij$. From Theorem \eqref{embedlima}, we know that $i$ is an embedding of a full subcategory.
\end{proof}
\begin{theorem}\label{charac} The embedding $i:\limA\hrar\kappa^a(\A)$ is an equivalence of categories.
\end{theorem}
\begin{proof}
We need to show that the embedding $i$ is essentially surjective. Let thus $X$ be an object of $\kappa^a(\A)$. By hypotheses, we are given an admissible (co-)cartesian square for all $i\leq i'\leq j\leq j'$:
$$
\xymatrix{ 
X_{i,j} \ar[d]
\ar[r]
& X_{i, j'}\ar[d] \\
X_{i', j}\ar[r]
& X_{i', j'}
}$$
Since the square is cartesian, and the mono and epimorphisms involved  are admissible, we obtain the admissible short exact sequence $X_{i,j}\rar X_{i,j'}\oplus X_{i',j}\rar X_{i',j'}$. Now take $i'=j$: Being $X_{j,j}=0$ we get  the admissible short exact sequence $X_{i,j}\rar X_{i,j'}\rar X_{j,j'}$ which for $i\leq j\leq j'$ expresses the admissibilty of the object $X$. Thus $X$ is in $\limA$, and the theorem is proved.
\end{proof}

\section{Exactness of the Beilinson category}

The goal of this section is to give a proof of the exactness of $\limA$, hence of the admissible Kato category $\kappa^a(\A)$, using the language of iterated ind/pro-objects. Thus, the admissible Kato category $\kappa^a(\A)$ of an exact category $\A$, introduced in Def. \eqref{admkato} can be thought of as the ``exact version" of the Kato category $\kappa(\A)$, which in general does not inherit in a canonical way an exact structure from that of $\A$. 

\begin{theorem}\label{exact} Let $(\A, \E)$ be an exact category. Then $\limA$ has a natural structure of an exact category.
\end{theorem}

We shall prove theorem \eqref{exact} by showing that $\limA$ is closed under extensions in the exact category $\Ind\Pro(\A)$. This will also give at once the exact structure of $\limA$. This amounts to show that if $X$ and $Z$ are objects in $\limA$ and there is an admissible short exact sequence in $\Ind\Pro(\A)$ given by $0\rar X\rar Y\rar Z$, the object $Y$ is isomorphic to an object of $\limA$.
\par
By using the straightification of the morphisms involved for the objects $X, \ Y, \  Z$ of $\Ind\Pro(\A)$, we can reduce ourselves to consider, in components, these objects such that for all $(i,j)$ we have admissible short exact sequences of $\A$:
\[
0\lrar X_{i,j}\hrar Y_{i,j}\epi Z_{i,j}\lrar 0
\]
compatible with the structure morphisms induced by the functors represented by $X, \ Y, \  Z$, and for which, whenever we set $i\leq i'$ and $j\leq j'$, we have that the squares
$$
\xymatrix{ 
X_{i,j} \ar[d]
\ar[r]
& X_{i, j'}\ar[d] \\
X_{i', j}\ar[r]
& X_{i', j'}
}$$
and
$$
\xymatrix{ 
Z_{i,j} \ar[d]
\ar[r]
& Z_{i, j'}\ar[d] \\
Z_{i', j}\ar[r]
& Z_{i', j'}
}$$
are both admissible and cartesian (hence cocartesian) in $\A$. In order to prove theorem \eqref{exact} it is enough to prove that also the square
\begin{equation}\label{ysquare}
\xymatrix{ 
Y_{i,j} \ar[d]
\ar[r]
& Y_{i, j'}\ar[d] \\
Y_{i', j}\ar[r]
& Y_{i', j'}
}
\end{equation}
is both admissible and cartesian, for all $(i,j)$ and $(i', j')$ in $\Pi$ such that $i\leq i'$ and $j\leq j'$.  In other words we want to prove that, under our assumptions, ``admissible cartesian squares are closed under extensions".
\par
Let's start by first showing that the square \eqref{ysquare} is an admissible one. We want thus to prove that its horizontal arrows are admissible monomorphisms, and its vertical ones admissible epimorphisms. For given $i\leq j\leq j'$, let's write the following commutative diagram:
$$
\xymatrix{ 
& 0 \ar[d]& 0 \ar[d] &0\ar[d] & \\
0\ar[r] & X_{i,j}\ar[d] \ar[r] &Y_{i,j} \ar[d]\ar[r] &Z_{i,j}\ar[d]\ar[r]& 0 \\
0\ar[r] & X_{i,j'}\ar[d] \ar[r] &Y_{i,j'} \ar[d]\ar[r] &Z_{i,j'}\ar[d]\ar[r]& 0 \\
0\ar[r] & X_{j,j'}\ar[d] \ar[r] &Y_{j,j'} \ar[d]\ar[r] &Z_{j,j'}\ar[d]\ar[r]& 0 \\
& 0 & 0 & 0  & \\
}
$$
in which the three rows are admissible short exact, and so are the first and the third column. But as we set this diagram in the abelian envelope $\F$, we can thus apply the middle $3\times 3$ lemma (see e.g. \cite{ml}, p. 208), and conclude that in $\F$ the middle sequence is also short exact. We thus apply the last part of Quillen's embedding theorem \eqref{embedding}, which tells us that this sequence in $\A$ is an admissible short exact sequence. This shows in a single shot that  in the square \eqref{ysquare}, the horizontal morphisms are admissible mono, as well as the vertical morphisms are admissible epi. This shows that the square \eqref{ysquare} is an admissible square.

\vspace{0.3cm}

It is then left to prove that the square \eqref{ysquare} is cartesian. This will follow from a diagram-chase argument involving an exact sequence of three squares (i.e. exact for the four sequences of the vertexes of the squares), whose first and last squares are cartesian.  We thus state this property as a proposition of homological algebra.

\begin{prop} Let be
\begin{equation}\label{first}
\xymatrix{ 
X' \ar[d]_{e'_2} 
\ar[r]^{m'_1}
& Y'\ar[d]^{e'_1} \\
T'\ar[r]_{m'_2}
& Z'
}
\end{equation}
\begin{equation}\label{second}
\xymatrix{ 
X \ar[d]_{e_2} 
\ar[r]^{m_1}
& Y\ar[d]^{e_1} \\
T\ar[r]_{m_2}
& Z
}
\end{equation}
\begin{equation}\label{third}
\xymatrix{ 
X'' \ar[d]_{e''_2} 
\ar[r]^{m''_1}
& Y''\ar[d]^{e''_1} \\
T''\ar[r]_{m''_2}
& Z''
}
\end{equation}
three admissible squares, where the morphisms indexed by $m$ are the 
admissible monos and those indexed by e the admissible epis. Suppose also that the squares \eqref{first} and \eqref{third}are cartesian in $\A$. Suppose that we have admissible short exact sequences: 

\begin{align}
X'\xrar{f} X\xrar{g} X'' &&  Y'\xrar{f'} Y\xrar{g'} Y'' \\
T'\xrar{f'''} T\xrar{g'''} T'' &&  Z'\xrar{f''} Z\xrar{g''} Z''
\end{align}
such that the resulting cubic diagram is commutative. Then, the square \eqref{second} is cartesian.
\end{prop}
\begin{proof} We shall use a diagram-chase argument, as we are in the hypotheses of the proposition \eqref{cartesian}. Thus, we can think of the above as a diagram of abelian groups, for which we are given elements $y\in Y$ and $t\in T$, such that $e_1(y)=m_2(t)$ in $Z$, and the claim is that there exists a unique element $a\in X$, such that $e_2(a)=t$ and $m_1(a)=y$.

\vspace{0.1cm}

Let be $\tilde x$ a preimage of $t$ through $e_2$.  Let's put $\tilde y=m_1(\tilde x)$.  Then we have $e_1(\tilde y-y)=0$. Consider now $g(\tilde x)\in X''$. Then $m''(g(\tilde x))\in Y''$, and we have: 
$e'_2g'(y-\tilde y)=g''e_1(y-\tilde y)=0$. Now, $e'_2g'(y-\tilde y)=0$, so that we obtain 
$e'_2g'(y)=e'_2g'(\tilde y)$.

\vspace{0.1cm}

From the last equality we thus get: $e'_2g'(y)=e'_2g'(\tilde y)=e'_2m''_1(g(\tilde x))=m''_2e''_2(g(\tilde x))=
m''_2g'''e_2(\tilde x)=m''_2g'''(t).$

\vspace{0.1cm}

In particular, the last equality implies that  the element $g'(y)\in Y''$ and the element $g'''(t)\in T''$, are such that $e'_2(g'(y))=m''_2(g'''(t))$ in $Z''$. But \eqref{third} is a cartesian square: then there exists a unique element $x''\in X''$ such that:
\begin{align}\label{*}
m''_1(x'')=g'(y) & & \text{and} & & e''_2(x'')=g'''(t).
\end{align}
Now, $g$ is onto: then we get a preimage of $x''$ via $g$ in $X$: $x''=g(x)$. Hence, $m_1(x)\in Y$ and 
$g'm_1(x)=m''_1g(x)=m_1''(x'')=g'(y).$ Then, $g'(m_1(x)-y)=0$, which yields: $m_1(x)-y\in\ker g'=im(f')$.

\vspace{0.1cm}

Let thus be $y'\in Y'$ such that $f'(y')=m_1(x)-y$. Consider now $e'_1(y')\in Z'$. We claim that $e'_1(y')$ has a preimage in $T'$, via $m_2'$.

\vspace{0.1cm}

In fact, consider the chain of equalities: $f''e'_1(y')=e_1f'(y')=e_1(m_1(x)-y)=e_1m_1(x)-e_1(y)=m_2e_2(x)-e_1(y)=m_2e_2(x)-m_2(t).$ We therefore get:
\begin{equation}\label{**}
f''e'_1(y')=m_2(e_2(x)-t).
\end{equation}
From it, we see that $f''e'_1(y')$ is in the image of $m_2$. Consider next $g'''(e_2(x)-t)=g'''e_2(x)-g'''(t)=e''_2(x'')-g'''(t)=0$, from \eqref{*}. It follows: $e_2(x)-t\in \ker g'''=im f'''$. We thus get an element $t'\in T'$ such that $f'''(t')=e_2(x)-t$.

\vspace{0.1cm}

We now compute $m'_2(t')$. We want to prove that $m'_2(t')=e'_1(y')$. Since $f''$ ins injective, it is sufficient to show that $f''m'_2(t')=f''e'_1(y')$.

\vspace{0.1cm}

From \eqref{**} we obtain: $f''m'_2(t')=m_2f'''(t')=m_2(e_2(x)-t)=f''e'_1(y')$. It follows indeed that $m'_2(t')=e'_1(y')$. So, we have found a pair of elements $t'$ and $y'$, such that $m'_2(t')=e'_1(y')$. Applying now the cartesianity of square \eqref{first} we therefore obtain a unique element $x'\in X'$ for which 
$m'_1(x')=y'$ and $e'_2(x')=t'$.

\vspace{0.1cm}

Finally, we consider the element $f(x')=x_0\in X$. We have: $m_1(x_0)=m_1f(x')=f'm'_1(x')=f'(y')=m_1(x)-y.$  As a consequence we can write: $y=m_1(x)-m_1(x_0)=m_1(x-x_0)$. We have thus found an element in $X$ which is sent to $y$ by $m_1$. If this element is sent to $t$ by $e_2$, the proof of the proposition is completed. Thus, let's calculate $e_2(x-x_0)$. It is:

\vspace{0.1cm}

$e_2(x-x_0)=e_2x-e_2x_0=e_2(x)-e_2f(x')=e_2(x)-f'''e'_2(x').$ Now, $m_2e_2(x-x_0)=m_2e_2(x)-m_2f'''e'_2(x')=m_2e_2(x)-f''m'_2e'_2(x')=m_2e_2(x)-f''e'_1m'_1(x')=m_2e_2(x)-f''e'_1(y')=m_2e_2(x)-e_1f'(y')=m_2e_2(x)-m_2e_2(x)+m_2(t)$, where the last equality comes from \eqref{**}. Hence, $m_2(e_2(x-x_0))=m_2(t).$ Being $m_2$ injective, it follows $e_2(x-x_0)=t$, and we are done. 
\end{proof}

\vspace{0.3cm}

As a corollary of the exactness of $\limA$ we can now prove that the categories $\Ind^a_{\aleph_0}(\A)$, $\Pro^a_{\aleph_0}(\A)$ are exact. In order to do so, we first describe precisely how one can embed $\Ind^a_{\aleph_0}(\A)$ and  $\Pro^a_{\aleph_0}(\A)$ in $\limA$. 

\vspace{0.2cm}

Let be $X$ an object of $\Ind^a_{\aleph}(\A)$. Suppose $X=\limXj$. As $X$ is a countable strictly admissible ind-object, we can assume $J=\Zee_+$. Let's also choose quotients $X_j/X_i$, for every admissible monomorphism $X_i\hrar X_j$ whenever $i\leq j$. Thus, let us define an object of $\limA$ out of $X$ as follows:
\begin{equation}\label{embedind}
\begin{split}
X_{0, j}& =X_j \\
X_{i, j}& =X_{0, j} \text{ for } i<0 \text{ and } j\geq 0 \\
X_{i,j}& =0 \text{ for } i<0 \text{ and } j<0 \\
X_{i,j}& =X_j/X_i \text{ for } 0\leq i\leq j.
\end{split}
\end{equation}
And the structure maps are defined in the obvious way. In particular, for $i<0$ and $j\geq 0$, the admissible epimorphisms $X_{i, j}\epi X_{i+1, j}$ are just identities. It is easy to see that that the equations \eqref{embedind} give an embedding 
$\Ind^a_{\aleph_0}(\A)\hrar \limA$, although the embedding depends on the choice of the quotients $X_j/X_i$. 

\begin{prop}\label{pre}
$\Ind^a_{\aleph_0}(\A)$, $\Pro^a_{\aleph_0}(\A)$ are exact subcategories of $\limA$.
\end{prop}

\begin{proof}
We will show that $\Ind^a_{\aleph_0}(\A)$ is closed under extensions in $\limA$, with a simplified version of the argument used to prove that $\limA$ is exact. Let be

\begin{equation}\label{indclosed}
X\hrar Y\epi Z
\end{equation}

an admissible short exact sequence in $\limA$, with $X, Z\in\Ind^a_{\aleph_0}(\A)$. We prove that $Y$ also is in $\Ind^a_{\aleph}(\A)$.

\vspace{0.1cm}

A priori, $Y$ is an object of $\limA$ (i.e. of $\Ind^a\Pro^a(\A)$). The proof consists in showing that $Y$ is trivial in its pro-component. Let be $Y=\limYij$. Then \eqref{indclosed} in components can be written, for all $(i, j)$, $i\leq j$, as an admissible short exact sequence of $\A$: 
$$
X_{i,j}\hrar Y_{i,j}\epi Z_{i,j}. 
$$

Here $X_{i,j}$ and $Z_{i,j}$ are defined as in \eqref{embedind}. Therefore, for $i<0$, $j<0$, we get the admissible short exact sequence
$$
0\hrar Y_{i,j}\epi 0
$$
which clearly forces $Y_{i,j}=0$ in this case. We are now left to prove that $Y_{i, j}=Y_{0, j}$ for $i<0$ and $j\geq 0$. Let's prove it for $i=-1$. From \eqref{embedind} we have $X_{-1, 0}=X_{0,0}$ and $Z_{-1, 0}=Z_{0, 0}$, and an admissible epimorphism $Y_{-1, 0}\epi Y_{0,0}$. We then have the commutative diagram
$$
\xymatrix{ 
X_{-1,0}\ar[d] \ar[r] &Y_{-1,0} \ar[d]\ar[r] &Z_{-1,0}\ar[d] \\
X_{0,0}\ar[r] &Y_{0,0} \ar[r] &Z_{0,0} \\
}
$$
the two rows are admissible short exact sequences, the first and the third  column are isomorphisms. From the Five Lemma (\cite{ml}, p. 205) it follows that also the morphism $Y_{-1, 0}\epi Y_{0,0}$ is then an isomorphism, and this proves that $Y$ is in $\Ind^a_{\aleph_0}(\A)$. Then it follows that $\Ind^a_{\aleph_0}(\A)$ is exact. Similarly one proves that $\Pro^a_{\aleph_0}(\A)$ is exact.
\end{proof}

\section{Examples; Tate spaces}

As an example of a category of the type $\limA$, we give the following

\begin{definition}\label{tatespaces}
Let be $k$ a field. The category $\T:=\dlim\vect$ is called the category of {\it Tate vector spaces} over $k$.
\end{definition}
The definition given here coincides with the one given by Arkhipov and Kremnizer in \cite{ak}. Definition \eqref{tatespaces} can be iterated since each category $\limA$ is exact for any exact category $\A$. We call the 
category $\T^n=\dlim^n\vect$ the category of {\it n-Tate spaces} over the 
field $k$.

\vspace{0.1cm}

The proof of the following is therefore clear:
\begin{prop}\label{limduality}
For any exact category $\A$, $(\limA)^o=\dlim(\A^o)$.
\end{prop}
In particular, being $\vect=\vect^o$, the category $\T$ is self-dual.

\vspace{0.3cm}

Let us denote by $\Ll_0$ the category of
linearly compact topological $k-$vector spaces and by $\Ll$ the category of locally linearly compact topological $k$-vector 
spaces and their morphisms, as introduced in \cite{l}, II.27.1 and II.27.9. 

\begin{lm}
There are equivalences of categories: 
$$
\Phi_0:\Pro(\vect)\xrar{\sim}\Pro^s(\vect)\xrar{\sim}\Ll_0.
$$
\end{lm}

\begin{proof}({\it Sketch.}) The first equivalence is proved by turning a projective system of finite-dimensional vector spaces $\proVj$ into a strict projective system $\proVvj$, by defining 
$$
V'_j\colon =\bigcap_{i\geq j}\im\{v_{ij}\colon V_i\to V_j\}
$$
where the $v_{ij}$'s are the structure morphisms of the pro-system of the $V_j$'s. 
Since the spaces $V_j$ are finite-dimensional, each intersection stabilizes; therefore the two pro-objects $\proVj$ and $\proVvj$ are isomorphic, and the latter object is  strict since its structure morphisms are surjections. The second equivalence is defined on the objects by 
$$
\proVvj\mapsto {\underset{j\in J}{\varprojlim}\  V'_j}
$$
And it is easy to see that the induced functor is full and faithful. The fact that it is also essentially surjective is a consequence of Theorem (32.1) of \cite{l}.
\end{proof}
In particular, the category $\Ll_0$ is an abelian category.
\begin{prop}\label{lefschetz}
There is an equivalence of categories: $\Phi:\T\xrar{\sim}\Ll$, whose restriction to the category $\Pro^s(\vect)$ is $\Phi_0$.
\end{prop}

\begin{proof}
It is naturally defined a functor $\Phi$ from $\T$ to the category of vector spaces over $k$, which takes the object $\limij V_{ij}$ into 
the space $\limitij\ V_{ij}$, and extended to the morphisms in the obvious way. The cartesian condition valid in $\T$ ensures that 
$\Phi$ takes values in the category $\Ll$, and it is not difficult to prove that $\Phi$ realizes the equivalence $\T\xrar{\sim}\Ll$.
\end{proof}

As a consequence of Proposition \eqref{lefschetz}, $\Ll$ becomes endowed with a structure of an exact category, and it is self-dual 
(see Prop.\eqref{limduality}). We now describe its exact structure in topological terms.

\begin{prop}\label{exactTate}
(a) Under the identification of Proposition \eqref{lefschetz}, the class of admissible monomorphisms of $\Ll$ coincides with the class 
of its closed embeddings. \\
(b) Similarly, the class of admissible epimorphisms in $\Ll$ coincides with the class of continuous surjective morphisms $p:B\rar C$, 
such that the canonical bijection $\dfrac{B}{\ker(p)}\rar C$ is a homeomorphism.
\end{prop}

\begin{proof}
(a). Let $\alpha:A\hrar B$ be an admissible monomorphism in $\T$. Let's prove that $\Phi(\alpha)\colon\Phi(A)\to\Phi(B)$ is a closed embedding.

\vspace{0.1cm}

Let be $A=\limij A_{ij}$, $B=\limij B_{ij}$. Straightify $\alpha$ and get a representation $\alpha=\limij\alpha_{ij}$, where for each 
$i, j$, $\alpha_{ij}:A_{ij}\hrar B_{ij}$ is an embedding of finite dimensional vector spaces. We thus have $\Phi(A)=\limitij\ A_{ij}$, 
$\Phi(B)=\limitij\ B_{ij}$, and $\Phi(\alpha)=\limitij\ \alpha_{ij}$. We write, for all $i$, $\alpha_i=\liminvj\ \alpha_ {ij}$, and 
$\beta_i=\Phi(\alpha_i)=\limitinvj\ \alpha_{ij}$. Let us also write $A'_i=\limitinvj\ A_{ij}$ and $B'_i=\limitinvj\ B_{ij}$. It is therefore 
$\beta_i\colon A'_i\to B'_i$ in $\Ll$, and $\Phi(\alpha)=\limiti\ \beta_i\colon\limiti\ A'_i\to\limiti\ B'_i$.

\vspace{0.2cm}

For all $i$, the map $\beta_i:\limitinvj\ A_{ij}\rar\limitinvj\ B_{ij}$ is an injective map of linearly compact topological vector spaces. 
From Prop. (27.8) of \cite{l}, it follows that $\beta_i$ is an isomorphism onto its image, i.e. an embedding. Next, the space $\beta_i(\limitinvj\ A_{ij})$ is 
linearly compact in $\limitinvj\ B_{ij}$, hence, from (27.5) of \cite{l}, closed in it. Therefore, $\beta_i$ is a closed embedding. The map 
$\limiti\ \beta_i$, from general properties of the functor $\varinjlim$, is thus  an embedding of locally linearly 
compact topological vector spaces. Since this map sends each linearly compact subspace $A_i$ into a closed subspace of 
$\limiti\ B_i$, it follows that it is closed. Then $\Phi(\alpha)$ is a closed embedding.

\vspace{0.2cm}

Conversely, let $\beta:V\hrar W$ be a closed embedding in $\Ll$, with $V=\Phi(A)$ and $W=\Phi(B)$. We can write: 
$V=\varinjlim\ U$ and $W=\varinjlim\ U'$, where the limits are taken over the set of all $U$ compact open, such that $U\subset V$  
in the first limit and the set of all $U'$ compact open such that $U'\subset W$ in the second. On the other hand, for $U$ and $U'$ so 
defined,  we have: $U=\underset{U_1\subset U}\varprojlim\ \dfrac{U}{U_1}$ and 
$U'=\underset{U'_1\subset U'}\varprojlim\ \dfrac{U'}{U'_1}$, and $\dfrac{U}{U_1}$ and $\dfrac{U'}{U'_1}$ are finite dimensional vector 
spaces. Since $\beta$ is a closed embedding, we can take $U=U'\cap V$ and $U_1=U'_1\cap V$. Thus, we can write $\beta$ as a 
limit of finite dimensional vector spaces as follows:
$$
\beta\colon\underset{U}\varinjlim\underset{U_1\subset U}\varprojlim\ \dfrac{U}{U_1}\hrar
\underset{U'}\varinjlim\underset{U'_1\subset U'}\varprojlim\ \dfrac{U'}{U'_1}
$$
such that each component of $\beta$ is a linear embedding of finite dimensional spaces $\dfrac{U}{U_1}\hrar\dfrac{U'}{U'_1}$. This 
component system corresponds, in the category $\T$, to an admissible monomorphism $\alpha:A\hrar B$, for which 
$\Phi(\alpha)=\beta$. Thus, part (a) of the theorem is proved.

\vspace{0.2cm}
The proof of (b) follows by duality from the proof of (a), and the theorem is proved.
\end{proof}

Proposition \eqref{lefschetz} allows us to identify $\T$ and $\Ll$.

\vspace{0.3cm}

The category $\T$ is not abelian. For example, the inclusion $k[t]\hrar k[[t]]$ is a non-admissible monomorphism in $\T$.

{\small Department Of Mathematics\\
Yale University\\
New Haven, CT 06520\\
\tt luigi.previdi@yale.edu}


\begin{thebibliography}{widest-label}
\bibitem{am} M. Artin, B. Mazur, \'Etale Homotopy, Lecture Notes in Mathematics 100, Springer-Verlag (1969).
\bibitem{ak} S. Arkhipov, K. Kremnizer, {\it 2-gerbes and 2-Tate spaces}, arXiv:0708.4401 (2007).
\bibitem{B} A. Beilinson, {\it How to glue perverse sheaves} in:  K-Theory, Arithmetic and Geometry, Yu. I. Manin (Ed.), Lecture Notes in Mathematics 1289, pp. 42-51, Springer-Verlag, New York (1987). 
\bibitem{dr} V. Drinfeld, {\it Infinite-Dimensional Vector Bundles in Algebraic 
Geometry}, in: The Unity of Mathematics, Birkhauser, Boston (2006).
\bibitem{gm} S. I Gelfand, Yu. I. Manin, Methods of Homological Algebra, Second Edition, Springer Monographs in Mathematics, Springer-Verlag, Berlin - Heidelberg (2003).
\bibitem{sga} A. Grothendieck et al., S\'eminaire de g\'eometrie alg\'ebrique IV:  Th\'eorie des topos et cohomologie \'etale des schemas. Lecture Notes in Mathematics 269, Springer-Verlag, Berlin - Heidelberg (1972).
\bibitem{ka} M. Kapranov, {\it Double affine Hecke algebras and 2-dimensional local fields.}  J. Amer. Math. Soc.  14  (2001),  no. 1, 239--262
\bibitem{ka2} M. Kapranov, {\it Infinite dimensional objects in algebra and geometry}. Course given at Yale University, Fall 2004.
\bibitem{kv} M. Kapranov, E. Vasserot, {\it Vertex algebras and the formal loop space}.  Publ. Math. Inst. Hautes Etudes Sci.   100  (2004), 209-269.
\bibitem{ks} M. Kashiwara, P. Schapira, {Categories and Sheaves}, Grundlehren der mathematischen Wissenschaften, Springer-Verlag, Berlin - Heidelberg (2006).
\bibitem{k}  K. Kato, {\it Existence theorem for higher local fields}, in: Geometry and Topology Monographs Vol 3: Invitation to higher local fields, I. Fesenko and M. Kurihara, (Eds.), pp. 165-195, University of Warwick (2000).
\bibitem{ke} B. Keller, {\it Derived categories and their uses.} Handbook of algebra, Vol. 1, 671-701, North-Holland, Amsterdam (1996). 
\bibitem{l} S. Lefschetz, Algebraic Topology (AMS Colloquium Publications 27), Amer. Math. Soc., New York (1942).
\bibitem{ml} S. Mac Lane, Categories for the Working Mathematician, Second Edition, GTM Springer Vol. 5, Springer-Verlag, New York (1998).
\bibitem{o} D. Osipov, {\it Adeles on $n-$Dimensional Schemes and categories $C_n$} Int. J. of Math. Vol. 18, no. 3 (2007), 269-279.
\bibitem{op} D. V. Osipov, A. N. Parshin, {\it Harmonic analysis on local fields and adelic spaces I},  arXiv:0707.1766 (2007).
\bibitem{p} L. Previdi, {\it Sato Grassmannians for generalized Tate spaces}, arXiv:1002.4863 (2010).
\bibitem{q} D. Quillen, {\it Higher Algebraic K-Theory I}, in: Higher K-Theories, Lecture Notes in Mathematics 341, pp. 77-139, Springer-Verlag, New York (1973).
\bibitem{rz} L. Ribes, P. Zalesskii, Profinite Groups, Springer-Verlag, Berlin-Heidelberg (2000).
\bibitem{w} F. Waldhausen, {\it Algebraic K-Theory of Generalized Free Products, Part I},  Annals of Mathematics  2nd Ser., vol 108, No. 2 (Sept. 1978), pp. 135-204.
\bibitem{we} C. Weibel, An Introduction to homological Algebra, Cambridge University Press (1995).

\end{thebibliography}
\end{document}